\documentclass{amsart}

\usepackage[english]{babel} 
\usepackage[utf8]{inputenc} 
\usepackage{amsmath} 
\usepackage{amsthm}
\usepackage{amssymb}
\usepackage{mathrsfs}
\usepackage{amsfonts} %
\usepackage{bm} 
\usepackage{fancyvrb}
\usepackage[linesnumbered,algoruled,boxed]{algorithm2e}
\usepackage{graphicx}
\usepackage{psfrag}
\usepackage{subfig}
\usepackage[square,sort,comma,numbers]{natbib} 
\usepackage[colorlinks=false, linkcolor=blue,  citecolor=OliveGreen,  pdfstartview={}{}]{hyperref} 
\usepackage{lineno}
\usepackage{stmaryrd}
\usepackage[font=scriptsize]{caption}
\usepackage{enumerate}

\DeclareFontEncoding{FMS}{}{}
\DeclareFontSubstitution{FMS}{futm}{m}{n}
\DeclareFontEncoding{FMX}{}{}
\DeclareFontSubstitution{FMX}{futm}{m}{n}
\DeclareSymbolFont{fouriersymbols}{FMS}{futm}{m}{n}
\DeclareSymbolFont{fourierlargesymbols}{FMX}{futm}{m}{n}
\DeclareMathDelimiter{\VERT}{\mathord}{fouriersymbols}{152}{fourierlargesymbols}{147}

\newtheorem{theorem}{Theorem}[section] 
\newtheorem{lemma}[theorem]{Lemma}
\theoremstyle{remark}
\newtheorem{remark}[theorem]{Remark} 
\numberwithin{equation}{section}
\newcommand{\T}{\mathscr{T}}

\newcommand{\Sides}{\mathscr{S}}

\newcommand{\E}{\mathscr{E}}

\usepackage[usenames,dvipsnames]{color}

\usepackage{pgf,tikz}
\usetikzlibrary{arrows}

\begin{document}


\title[Error estimates for an optimal control problem in measure space]{An a posteriori error analysis of an elliptic optimal control problem in measure space}
\thanks{EO is partially supported by CONICYT through FONDECYT project 3160201. AJS is partially supported by NSF grant DMS-1720213.}






\author{Francisco Fuica\textsuperscript{\textdagger}}
\address{\textdagger Departamento de Matem\'atica, Universidad T\'ecnica Federico Santa Mar\'ia, Valpara\'iso, Chile.
\texttt{francisco.fuica@sansano.usm.cl}}
\author{Enrique Ot\'arola\textsuperscript{\textdaggerdbl}}\address{\textdaggerdbl Departamento de Matem\'atica, Universidad T\'ecnica Federico Santa Mar\'ia, Valpara\'iso, Chile.
\texttt{enrique.otarola@usm.cl}}
\author{Abner J.~Salgado\textsuperscript{\textsection}}
\address{\textsection Department of Mathematics, University of Tennessee, Knoxville, TN 37996, USA. \texttt{asalgad1@utk.edu}
}

\begin{abstract}
We propose an
a posteriori error estimator for a sparse optimal control problem: the control variable lies in the space of regular Borel measures. We consider a solution technique that relies on the discretization of the control variable as a linear combination of Dirac measures. The proposed a posteriori error estimator can be decomposed into the sum of two contributions: an error estimator in the maximum norm for the discretization of the adjoint equation and an estimator in the $L^2$--norm that accounts for the approximation of the state equation. We prove that the designed error estimator is locally efficient and we explore its reliability properties. The analysis is valid for two and three--dimensional domains. We illustrate the theory with numerical examples.
\end{abstract}
\maketitle

\section{Introduction}\label{introduction}
This work is dedicated to the design and analysis of an
efficient a posteriori error estimator for a sparse elliptic optimal control problem: the control variable is sought in the space of regular Borel measures. To make matters precise, for $d\in \{2,3\}$, we let $\Omega\subset\mathbb{R}^{d}$ be an open, bounded and convex polytopal domain; the boundary of $\Omega$ is denoted by $\partial \Omega$. Given a desired state $y_d\in L^2(\Omega)$, and a sparsity parameter $\alpha > 0$, we introduce the cost functional
\begin{equation}
J(y,u)
:=
\frac{1}{2}\|y-y_d\|_{L^2(\Omega)}^2 
+
\alpha \|u\|_{\mathcal{M}(\Omega)},
\end{equation}
and thus define our sparse optimal control problem as follows: Find
\begin{equation}\label{eq:continuous_opc}
\min 
J(y,u)
\end{equation}
subject to the linear and elliptic PDE
\begin{equation}\label{eq:elliptic_pde}
-\Delta y = u \text{ in } \Omega, 
\qquad
y = 0 \text{ on  } \partial \Omega.
\end{equation}

Notice that the control variable $u$ lies in the space of regular Borel measures $\mathcal{M}(\Omega)$.

The design and analysis of solution techniques for optimal control problems that induce a sparse structure in the control variable have been widely studied in the literature over the last decade. The first work that provides an analysis for this class of problems is \citep{Stadler2009}; the sparsity arises from the consideration of a $L^1(\Omega)$--control cost term in the quadratic cost functional. The author of \cite{Stadler2009} studied a regularized problem, derived optimality conditions, proposed and analyzed a semismooth Newton method. Solution techniques based on finite element methods have been proposed and analyzed in \citep{WW:11} when the state equation is linear and in \citep{CHW2012,CHW:12,CHW:12b} when the state equation is a semilinear elliptic PDE. For an up--to--date overview of the theory we refer the interested reader to \cite{Casas2017}. The extension of the theory to the evolutionary case has been recently explored in \citep{MR3780467,MR3601024,MR3612174,MR3669663}.
%

Regarding the optimal control problem \eqref{eq:continuous_opc}--\eqref{eq:elliptic_pde}, and to the best of our knowledge, the first work that provided an analysis was \citep{CK2011}. In this work, the authors analyzed elliptic control problems with measures and functions of bounded variation as controls; existence and uniqueness of the corresponding predual problems were discussed together with the solution of the optimality systems by a semismooth Newton method. Subsequently, 
in \citep{CK2012}, the authors address the feasibility of optimal source placement by optimal control in measure spaces: they  extend \citep{CK2011} by including partial observation, control on subdomains and non--negativity properties of the controls. A numerical scheme based on finite element techniques was later proposed in \citep{CCK2012}. 
The space $\mathcal{M}(\Omega)$ was discretized as the set of linear combinations of Dirac masses at interior mesh points.
The authors proved convergence of the scheme and provided error estimates. On the basis of this discretization scheme, improved error estimates were obtained in \citep{PV2013}. However, these error estimates, are not optimal in terms of approximation. This is due to the fact that the state variable exhibits reduced regularity properties.
It is thus just natural to propose adaptive finite element methods (AFEMs) to efficiently resolve the optimal control problem \eqref{eq:continuous_opc}--\eqref{eq:elliptic_pde} and recover optimal rates of convergence for the state variable.

AFEMs are a fundamental numerical instrument in science and engineering that allows for the resolution of PDEs with relatively modest computational resources. They are known to outperform classical FEM in practice and deliver optimal convergence rates when the latter cannot. To extract the local errors incurred by FEM, and thus be able to equidistribute them, AFEMs rely on \emph{a posteriori error estimators}, which are computable quantities, that depend on the discrete solution and data.
In contrast to the well established theory for linear elliptic PDEs, the design and analysis of a posteriori error estimators for optimal control problems are being currently developed. In view of their inherent nonlinear feature the analysis involves more arguments and technicalities.  To the best of our knowledge, the only work that provides an advance concerning the a posteriori error analysis for \eqref{eq:continuous_opc}--\eqref{eq:elliptic_pde} is \citep{CKW2016}. In this reference, the authors propose a functional error estimator, and prove that its square root yields an upper bound for the approximation error of the state variable and the error between the discrete and continuous cost functionals \citep[Section 5]{CKW2016}; an efficiency analysis, however, is not provided. 

In light of the discussion given above, the main objective of this work is to propose and analyze 
an efficient a posteriori error estimator for the optimal control problem \eqref{eq:continuous_opc}--\eqref{eq:elliptic_pde}.
We consider a solution technique for \eqref{eq:continuous_opc}--\eqref{eq:elliptic_pde} that relies on the discretization of the state and adjoint variables with piecewise linear functions, whereas the control variable is discretized using the framework presented in \citep[Section 3]{CCK2012}. The proposed a posteriori error estimator only accounts for the discretization of the state and adjoint variables. We measure the error of the state variable in $L^2(\Omega)$--norm and the error of the adjoint variable in $L^\infty(\Omega)$--norm, and we derive 
local efficiency results. We also explore the reliability properties of the designed error estimator. On the basis of the constructed a posteriori error estimator, we also design simple adaptive strategies that yield optimal rates of convergence for the numerical examples that we perform. We would like to mention that our error estimator is simpler than the one considered in \cite[Section 5]{CKW2016}. In addition, and in contrast to \cite[Section 5]{CKW2016}, the error indicator that we consider for the adjoint variable in the $L^{\infty}(\Omega)$--norm allows for unbounded forcing terms. This is of importance since, as it can be be observed from \eqref{eq:adjoint_pde}, the adjoint equation has $\bar y - y_d$ as a forcing term and, in general, $\bar y - y_d \notin L^{\infty}(\Omega)$. Additional assumptions must be imposed on $y_d$ in order to have that $\bar y - y_d \in L^{\infty}(\Omega)$ \cite[Theorem 2.5]{PV2013}.

The outline of this manuscript is as follows. In Section \ref{opc_in_measure} we present existence and uniqueness results together with first--order optimality conditions. In Section \ref{FEM} we present the finite element discretization of our optimal control problem; it relies on the discretization of the state and adjoint equations by using piecewise linear functions, whereas the control variable is approximated with Dirac deltas. The a posteriori error analysis of elliptic problems with delta sources together with maximum--norm a posteriori error estimation of elliptic problems are reviewed in Section \ref{laplacian_afem}. The core of our work is Section \ref{global_estimator}, where we design an a posteriori error estimator and study reliability and efficiency results. We conclude, in Section \ref{ne}, with a series of numerical examples that illustrate our theory.

Throughout this work $d \in \{ 2,3\}$. If $\mathcal{X}$ and $\mathcal{Y}$ are normed vector spaces, we write $\mathcal{X}  \hookrightarrow \mathcal{Y}$ to denote that $\mathcal{X}$ is continously embedded in $\mathcal{Y}$. We denote by $\mathcal{X}'$ and $\|\cdot\|_{\mathcal{X}}$ the dual and the norm of $\mathcal{X}$, respectively. If $X$ is a function space over the domain $G \subset \mathbb{R}^d$, we denote by $\langle \cdot, \cdot \rangle_{G}$ the duality pairing between $X$ and $X'$. The relation $a \lesssim b$ indicates that $a \leq C b$, with a nonessential constant $C$ that might change at each occurrence.

\section{The optimal control problem in measure space}\label{opc_in_measure}

In this section we review some of the main results related to the existence and uniqueness of solutions for problem \eqref{eq:continuous_opc}--\eqref{eq:elliptic_pde}. In addition, we present first--order necessary and sufficient optimality conditions.

We recall that the space of regular Borel measures $\mathcal{M}(\Omega)$ can be identified, by the Riesz Theorem, with the dual of the space of continuous functions that vanish on the boundary $\partial \Omega$, which from now on we shall denote by $C_0(\Omega)$. Given a measure $\mu \in \mathcal{M}(\Omega)$, we have that \begin{equation}\label{measure_norm}
\|\mu\|_{\mathcal{M}(\Omega)}
=
\sup_{{{\varphi\in C_0(\Omega)}\atop{\|\varphi\|_{C_0(\Omega)}\leq 1}}}\langle \mu, \varphi \rangle_\Omega
=
\sup_{{\varphi\in C_0(\Omega)}\atop{\|\varphi\|_{C_0(\Omega)}\leq 1}} \int_\Omega \varphi \:\text{d}\mu.
\end{equation}

Given $u\in \mathcal{M}(\Omega)$, we define the weak solution of problem \eqref{eq:elliptic_pde} as follows: 
\begin{equation}\label{eq:weak_pde}
y \in W_0^{1,r}(\Omega): \quad (\nabla y, \nabla v)_{L^2(\Omega)}
=
\langle u, v \rangle_\Omega \qquad \forall v \in W_0^{1,r'}(\Omega),
\end{equation}  
where $1\leq r < d/(d-1)$ and $r'$ denotes the conjugate exponent of $r$. Problem \eqref{eq:weak_pde} has a unique solution $y \in W_0^{1,r}(\Omega)$ that satisfies \citep[Theorem 4]{casas1986}, \cite[Th\'eor\`eme 9.1]{MR0192177}
\[
 \|  \nabla y \|_{L^{r}(\Omega)} \lesssim \| u \|_{\mathcal{M}(\Omega)}.
\]

We notice that, in view of the fact that $W_0^{1,r}(\Omega) \hookrightarrow L^2(\Omega)$ for $2d/(d+2) \leq r < d/(d-1)$, the cost functional $J$, which is defined by \eqref{eq:continuous_opc}, is well--defined. 

The following existence result follows from \citep[Proposition 2.2]{CK2011}.

\begin{theorem}[existence and uniqueness]
\label{well_pd_opc}
The sparse optimal control problem \eqref{eq:continuous_opc}--\eqref{eq:elliptic_pde} has a unique solution $(\bar{y},\bar{u})\in W_0^{1,r}(\Omega) \times \mathcal{M}(\Omega)$ for $1\leq r < d/(d-1)$.
\end{theorem}

The next result establishes optimality conditions for problem \eqref{eq:continuous_opc}--\eqref{eq:elliptic_pde}; see \citep[Theorem 2.1]{CCK2012} and \citep[Theorem 2.2]{PV2013}.

\begin{theorem}[optimality conditions]
\label{opt_cond}
Let $(\bar y,\bar u) \in W^{1,r}_0(\Omega) \times \mathcal{M}(\Omega)$ be the (unique) solution of \eqref{eq:continuous_opc}--\eqref{eq:elliptic_pde}. Then,
there exists a unique element $\bar{p}\in H^2(\Omega)\cap H_0^1(\Omega)$ satisfying
\begin{equation}\label{eq:adjoint_pde}
(\nabla w, \nabla \bar{p})_{L^2(\Omega)} = (\bar{y}-y_d, w)_{L^2(\Omega)}
\quad
\forall w\in H_0^1(\Omega),
\end{equation}
such that, for all $u\in \mathcal{M}(\Omega)$,
\begin{equation}\label{eq:opt_cond}
-\langle u - \bar{u}, \bar{p}\rangle_\Omega + \alpha \|\bar{u}\|_{\mathcal{M}(\Omega)}
\leq 
\alpha \| u \|_{\mathcal{M}(\Omega)}.
\end{equation}
In addition,
\begin{equation}
\|\bar{p}\|_{C_0(\Omega)} = \alpha \textrm{ if } \bar u \neq 0, \qquad 
\|\bar{p}\|_{C_0(\Omega)} \leq \alpha \textrm{ if } \bar u = 0.
\end{equation}
\end{theorem}

As a conclusion, the pair $(\bar{y},\bar{u})\in W^{1,r}_0(\Omega) \times \mathcal{M}(\Omega)$ is optimal for \eqref{eq:continuous_opc}--\eqref{eq:elliptic_pde} if and only if the triplet $(\bar{y},\bar{p},\bar{u}) \in W_0^{1,r}(\Omega) \times H_0^1(\Omega) \times \mathcal{M}(\Omega) $ satisfies the following optimality system:
\begin{equation}\label{eq:optimality_system}
\begin{cases}
\begin{array}{cl}
\quad(\nabla \bar{y}, \nabla v)_{L^2(\Omega)}  =  \langle\bar u, v\rangle_\Omega
\quad \quad \quad \qquad & \forall v\in W_{0}^{1,r'}(\Omega),\\
(\nabla w, \nabla \bar{p} )_{L^2(\Omega)}  =  (\bar{y}-y_{\Omega},w)_{L^2(\Omega)}
& \forall w\in H_{0}^{1}(\Omega) ,\\
\multicolumn{1}{c}{-\langle u - \bar{u}, \bar{p}\rangle_\Omega + \alpha \|\bar{u}\|_{\mathcal{M}(\Omega)}
\leq 
\alpha \| u \|_{\mathcal{M}(\Omega)}} &  \forall u\in \mathcal{M}(\Omega).
\end{array}
\end{cases}
\end{equation}  


\section{Finite element discretization}\label{FEM}

We recall the finite element approximation of the sparse optimal control problem \eqref{eq:continuous_opc}--\eqref{eq:elliptic_pde} developed in \citep[Section 3]{CCK2012}. In addition, we present convergence rates for such a discretization \cite{PV2013}.

We begin by introducing some ingredients of standard finite element approximation \cite{CiarletBook,Guermond-Ern}. Let $ \mathscr{T} = \{T\}$ be a conforming partition of $\bar \Omega$ into simplices $T$ with size $h_T := \textrm{diam}(T)$, and set $h_{\mathscr{T}}:= \max_{ T \in \mathscr{T}} h_T$. Let us denote by $\mathbb{T}$ the collection of conforming and shape regular meshes that are refinements of $\mathscr{T}_0$, where $\mathscr{T}_0$ represents an initial mesh. Given $\mathscr{T} \in \mathbb{T}$, we denote by $N_\mathscr{T}$ the number of interior nodes of $\mathscr{T}$ and by $\{x_i\}_{i=1}^{N_\mathscr{T}}$ the set of interior nodes of $\mathscr{T}$.

Given a mesh $\mathscr{T} \in \mathbb{T}$, we define the finite element space of continuous piecewise polynomials of degree one as
\begin{equation}
\label{def:piecewise_linear_space}
\mathbb{V}(\mathscr{T})=\{v_\mathscr{T}\in C_0(\Omega): v_{\mathscr{T}}|_{T}\in \mathbb{P}_1(T) \, \forall T\in \mathscr{T} \}.
\end{equation}
Given a node $x_i$ in the mesh $\mathscr{T}$, we introduce the function $\phi_i \in \mathbb{V}(\mathscr{T})$, which is such that $\phi_{i}(x_j) = \delta_{ij}$ for all $j = 1,\ldots,N_\mathscr{T}$. The set $\{\phi_i\}_{i=1}^{N_\mathscr{T}}$ is the so--called Courant basis of the space $\mathbb{V}(\mathscr{T})$.

With this setting at hand, we define the following discrete version of the optimal control problem \eqref{eq:continuous_opc}--\eqref{eq:elliptic_pde}: Find $(y_{\T}, u) \in \mathbb{V}(\T) \times \mathcal{M}(\Omega)$ that minimizes
\begin{equation}\label{eq:discrete_opc}
J(y_\mathscr{T},u)
\end{equation}
subject to the discrete state equation
\begin{equation}\label{eq:discrete_pde}
(\nabla y_\mathscr{T},\nabla v_\mathscr{T})_{L^2(\Omega)}
=
\langle u, v_\mathscr{T}\rangle_\Omega \qquad \forall v_\mathscr{T}\in \mathbb{V}(\mathscr{T}).
\end{equation}
Notice that, since the control variable is not discretized, the solution technique \eqref{eq:discrete_opc}--\eqref{eq:discrete_pde} corresponds to an instance of the so--called variational discretization approach \citep{hinze2005}.

The discrete optimal control problem \eqref{eq:discrete_opc}--\eqref{eq:discrete_pde} admits a solution.
In contrast to the continous case,  the discrete version of the control--to--state map $S_\mathscr{T}:  \mathcal{M}(\Omega) \ni u \mapsto y_\mathscr{T} \in \mathbb{V}(\mathscr{T})$ is not injective \cite{CCK2012,PV2013}. This and the fact that the norm 
$\| \cdot \|_{\mathcal{M}(\Omega)}$ is not strictly convex imply that $J$ is not strictly convex.
As a consequence, the uniqueness of the control variable cannot be guaranteed. However, among all the possible optimal controls, a special solution can be explicitly characterized using a particular discrete space \citep[Section 3]{CCK2012}: 
\begin{equation}\label{def:discrete_control}
\mathbb{U}(\mathscr{T}):= \left \{ u_\mathscr{T} \in \mathcal{M}(\Omega) \: :\: u_\mathscr{T}=\sum_{i=1}^{N_\mathscr{T}} u_i \delta_{x_i}, \quad u_i\in \mathbb{R},\quad  1 \leq i \leq N_\mathscr{T}  \right\}.
\end{equation}
Notice that this discrete space consists of linear combinations of Dirac measures associated to the interior nodes of the mesh $\mathscr{T}$. Finally, we introduce the following operator \citep[Section 3]{CCK2012}:
\begin{equation}\label{def:operator_lambda}
\Lambda_\mathscr{T} : \mathcal{M}(\Omega) \rightarrow \mathbb{U}(\mathscr{T}),\qquad \Lambda_\mathscr{T}(u)=\sum_{i=1}^{N_\mathscr{T}} \langle u, \phi_i\rangle_\Omega \delta_{x_i}.
\end{equation}

With the previous definitions at hand, we are in position to present the following results related to the convergence of the discrete solutions; see \citep[Theorems 3.2 and 3.5]{CCK2012} and \cite[Theorem 3.1]{PV2013}.

\begin{theorem}[convergence]
\label{discrete_existence_sol}
Among all the optimal controls of problem \eqref{eq:discrete_opc}--\eqref{eq:discrete_pde}, there exists a unique $\bar u_{\mathscr{T}} \in \mathbb{U}(\mathscr{T})$. Any other optimal control $\tilde{u}_\mathscr{T}\in \mathcal{M}(\Omega)$ of  \eqref{eq:discrete_opc}--\eqref{eq:discrete_pde} satisfies that $\Lambda_\mathscr{T}\tilde{u}_\mathscr{T} = \bar{u}_\mathscr{T}$. In addition, we have the following convergence properties as $h_{\mathscr{T}} \rightarrow 0$:
\begin{equation}\label{eq:basic_convergence}
 \bar{u}_\mathscr{T} \stackrel{*}{\rightharpoonup} \bar{u} \text{ in } \mathcal{M}(\Omega), \qquad \|\bar{u}_\mathscr{T}\|_{\mathcal{M}(\Omega)} \rightarrow \|\bar{u}\|_{\mathcal{M}(\Omega)},
\quad
\| \bar y - \bar y_{\T} \|_{L^2(\Omega)} \rightarrow 0.
\end{equation}
\end{theorem}

The following results present optimality conditions for the discrete optimal control problem \eqref{eq:discrete_opc}--\eqref{eq:discrete_pde}; see \citep[Theorem 3.2]{PV2013}.

\begin{theorem}[discrete optimality conditions]
\label{discrete_opt_cond} 
Let $(\bar{y}_\mathscr{T},\bar{u}_\mathscr{T})\in \mathbb{V}(\mathscr{T}) \times \mathbb{U}(\mathscr{T})$ be the discrete solution, as in Theorem \ref{discrete_existence_sol}. Then there exists a unique discrete adjoint state $\bar{p}_\mathscr{T}\in \mathbb{V}(\mathscr{T})$ such that 
\begin{equation}\label{eq:discrete_adj_st}
(\nabla w_\mathscr{T}, \nabla \bar{p}_\mathscr{T})_{L^2(\Omega)}
=
(\bar{y}_\mathscr{T}-y_d, w_\mathscr{T})_{L^2(\Omega)}
\quad \forall w_\mathscr{T}\in \mathbb{V}(\mathscr{T}),
\end{equation}
and that satisfies 
\begin{equation}\label{eq:discrete_opt_cond}
-\langle u - \bar{u}_\mathscr{T}, \bar{p}_\mathscr{T}\rangle_\Omega + \alpha \|\bar{u}_\mathscr{T}\|_{\mathcal{M}(\Omega)}
\leq 
\alpha \| u \|_{\mathcal{M}(\Omega)} \quad \forall u\in \mathcal{M}(\Omega).
\end{equation}
\end{theorem}

Consequently, the discrete pair $(\bar{y}_\mathscr{T},\bar{u}_\mathscr{T}) \in \mathbb{V}(\mathscr{T}) \times \mathbb{U}(\mathscr{T})$ is optimal for \eqref{eq:discrete_opc}--\eqref{eq:discrete_pde} if and only if the triplet $(\bar{y}_\mathscr{T},\bar{p}_\mathscr{T},\bar{u}_\mathscr{T}) \in \mathbb{V}(\mathscr{T}) \times \mathbb{V}(\mathscr{T}) \times \mathbb{U}(\mathscr{T})$ solves
\begin{equation}\label{eq:discrete_optimality_system}
\begin{cases}
\begin{array}{cl}
\quad(\nabla \bar{y}_\mathscr{T}, \nabla v_\mathscr{T})_{L^2(\Omega)}  =  \langle\bar u_\mathscr{T}, v_\mathscr{T}\rangle_\Omega
\quad \quad \quad \qquad & \forall v_\mathscr{T}\in \mathbb{V}(\mathscr{T}),\\
(\nabla w_\mathscr{T}, \nabla \bar{p}_\mathscr{T} )_{L^2(\Omega)}  =  (\bar{y}_\mathscr{T}-y_{\Omega},w_\mathscr{T})_{L^2(\Omega)}
& \forall w_\mathscr{T}\in \mathbb{V}(\mathscr{T}) ,\\
\multicolumn{1}{c}{-\langle u - \bar{u}_\mathscr{T}, \bar{p}_\mathscr{T}\rangle_\Omega + \alpha \|\bar{u}_\mathscr{T}\|_{\mathcal{M}(\Omega)}
\leq 
\alpha \| u \|_{\mathcal{M}(\Omega)}} &  \forall u\in \mathcal{M}(\Omega).
\end{array}
\end{cases}
\end{equation}  

To conclude this section, we present a priori error estimates for the approximation of the optimal state variable. To state it,
we will need an extra assumption on the desired state $y_d$; see \citep[Section 4]{PV2013}. Let us assume that
\begin{equation}\label{assumption}
y_d \in  L^\infty(\Omega) \textrm{ for } d = 2, \quad y_d \in L^3(\Omega) \textrm{ for } d = 3.
\end{equation}
\begin{theorem}[a priori error estimates]\label{state_convergence}
Let $y_d$ satisfy \eqref{assumption}. Let $(\bar{y},\bar{u})\in W_0^{1,r}(\Omega) \times \mathcal{M}(\Omega)$ be the solution of problem \eqref{eq:continuous_opc}--\eqref{eq:elliptic_pde}, and let $(\bar{y}_\mathscr{T},\bar{u}_\mathscr{T})\in \mathbb{V}(\mathscr{T})\times \mathbb{U}(\mathscr{T}) $ the solution of problem \eqref{eq:discrete_opc}--\eqref{eq:discrete_pde}, given as in Theorem \ref{discrete_existence_sol}. Then there holds
\begin{align}\label{eq:improved_estimate_1}
\|\bar{y}-\bar{y}_\T\|_{L^2(\Omega)} 
& \lesssim
h^{2-d/2}_\mathscr{T}|\ln{h_\mathscr{T}}|^\frac{\gamma}{2},
\\
\|\bar{u}-\bar{u}_\T\|_{H^{-2}(\Omega)} 
& \lesssim 
h^{2-d/2}_\mathscr{T}|\ln{h_\mathscr{T}}|^\frac{\gamma}{2}, \label{eq:improved_estimate_2}
\end{align}
with $\gamma=\frac{7}{2}$ for $d=2$, and $\gamma=1$ for $d=3$.
\end{theorem}
\begin{proof}
We refer the reader to \citep[Theorem 4.4]{PV2013} for a proof of \eqref{eq:improved_estimate_1} and \citep[Corollary 4.5]{PV2013} for \eqref{eq:improved_estimate_2}.
\end{proof}

\begin{theorem}[a priori error estimates]\label{state_convergence_2}
Let $y_d \in L^{\infty}(\Omega)$, which implies that the solution $(\bar{y},\bar{u})\in W_0^{1,r}(\Omega) \times \mathcal{M}(\Omega)$ of problem \eqref{eq:continuous_opc}--\eqref{eq:elliptic_pde} satisfies that $\bar y \in H_0^1(\Omega) \cap L^{\infty}(\Omega)$. If $(\bar{y}_\mathscr{T},\bar{u}_\mathscr{T})\in \mathbb{V}(\mathscr{T})\times \mathbb{U}(\mathscr{T}) $ denote the solution of problem \eqref{eq:discrete_opc}--\eqref{eq:discrete_pde}, given as in Theorem \ref{discrete_existence_sol}, then
\begin{equation}
\label{eq:improved_estimate_3}
\|\bar{y}-\bar{y}_\T\|_{L^2(\Omega)} 
\lesssim
 h_\mathscr{T}|\ln{h_\mathscr{T}}|^\frac{\rho}{2},
\end{equation}
with $\rho=2$ for $d=2$ and $\rho=11/4$ for $d=3$.
\end{theorem}
\begin{proof}
We refer the reader to \citep[Theorem 5.1]{PV2013}.
\end{proof}


\section{A posteriori error analysis for the Laplacian}\label{laplacian_afem}
In the next section we will construct an a posteriori error estimator for the sparse optimal control problem \eqref{eq:continuous_opc}--\eqref{eq:elliptic_pde} that will be based on two error contributions: one associated to the discretization of the state equation \eqref{eq:weak_pde} and another one related to the discretization of the adjoint equation \eqref{eq:adjoint_pde}. In order to design these contributions, and in an effort to make the presentation of the material as clear as possible, in this section we briefly review a posteriori error estimates for Poisson problems. We first review the $L^2$ a posteriori error estimator, developed in \cite{ABR2006}, for a Poisson problem that involves a Dirac measure as a source term, and then the pointwise a posteriori error estimator of \cite{AORS2017} for a Poisson problem involving an unbounded forcing term. The latter is needed because the adjoint problem \eqref{eq:adjoint_pde} has the function $\bar y - y_d$ as a forcing term, which in general does not belong to $L^{\infty}(\Omega)$. In order to have that 
$\bar y - y_d \in L^{\infty}(\Omega)$, an additional assumption must be imposed: $y_d \in L^{\infty}(\Omega)$; see \cite[Theorem 2.5]{PV2013}.

\subsection{A posteriori  error estimates for the Laplacian with Dirac sources}
\label{dirac_estimator}
Let $\xi$ be an interior point of $\Omega$ and consider the following elliptic boundary value problem: Find $z$ such that
\begin{equation}
\label{eq:dirac_eq}
-\Delta z=\delta_{\xi} \textrm{ in } \Omega,\quad z=0 \textrm{ on } \partial\Omega.
\end{equation}
Consider the following weak formulation of problem \eqref{eq:dirac_eq}:
\begin{equation}
\label{eq:weak_dirac_eq}
z\in W_0^{1,r}(\Omega): \quad
(\nabla z, \nabla v)_{L^2(\Omega)} = \delta_{\xi}(v) \quad \forall v\in W_0^{1,r'}(\Omega),
\end{equation}
where $1\leq r < d/(d-1)$ and $r'$ denotes its conjugate exponent. We immediately notice that since $r' > d$, we have that $W^{1,r'}(\Omega) \hookrightarrow C(\bar \Omega)$ and consequently, the term on the right--hand side of \eqref{eq:weak_dirac_eq} is well--defined.

We now define the Galerkin approximation to \eqref{eq:weak_dirac_eq} as the solution to the following problem: 
\begin{equation}
\label{eq:discrete_dirac}
z_\mathscr{T} \in \mathbb{V}(\mathscr{T}): \quad (\nabla z_\mathscr{T},\nabla v_\mathscr{T})_{L^2(\Omega)}=\delta_{\xi}(v_\mathscr{T})\quad\forall\:v_\mathscr{T}\in \mathbb{V}(\mathscr{T}),
\end{equation}
where the discrete space $\mathbb{V}(\mathscr{T})$ is defined in \eqref{def:piecewise_linear_space}.

In order the present the error estimator developed in \cite{ABR2006}, we introduce standard notation in a posteriori error analysis \cite{ver2013}. We define $\Sides$ as the set of internal $(d-1)$--dimensional interelement boundaries $S $ of $\T$. For $T \in \T$, let $\Sides_T$ denote the subset of $\Sides$ that contains the sides in $\Sides$ which are sides of $T$. We also denote by $\mathcal{N}_{S}$ the subset of $\T$ that contains the two elements that have $S$ as a side. In addition, we define the following patches associated with an element $T \in \T$:
\begin{equation}\label{NT}
\mathcal{N}_T:= \bigcup_{T' \in \T: T \cap T' \neq \emptyset} T',
\end{equation}
and
\begin{equation}\label{NT_star}
\mathcal{N}_T^*:= \bigcup_{T' \in \T: \Sides_T \cap \Sides_{T'} \neq \emptyset} T'.
\end{equation}

Given a discrete function $z_{\T} \in \mathbb{V}(\T)$, we define, for any internal side $S \in \Sides$,
the jump or interelement residual  $[\![ \nabla z_\mathscr{T}\cdot \nu ]\!]$ by
\begin{equation}\label{def:jump}
[\![ \nabla z_\mathscr{T}\cdot \nu ]\!]= \nu^{+} \cdot \nabla z_{\mathscr{T}}|_{T^{+}} + \nu^{-} \cdot \nabla z_{\mathscr{T}}|_{T^{-}},
\end{equation}
where $\mathcal{N}_S = \{ T^+, T^-\}$ and $\nu^{+}, \nu^{-}$ denote the unit normals to $S$ pointing towards $T^{+}$, $T^{-} \in \T$, respectively.

With these ingredients at hand, we define the following a posteriori error indicators. If $\xi \notin \{x_i\}_{i=1}^{N_\mathscr{T}}$ but $\xi \in T$, then
\begin{equation}\label{def:basic_estimator1}
\E_{\delta}^2(z_\T;T):=  h_T^{4-d} \chi(T) + 
h_T^{3}\|[\![ \nabla z_\mathscr{T}\cdot \nu ]\!]\|_{L^2(\partial T \setminus \partial \Omega)}^2,
\end{equation}
otherwise,
\begin{equation}\label{def:basic_estimator2}
\E_{\delta}^2(z_\T;T):=  h_T^{3}\|[\![ \nabla z_\mathscr{T}\cdot \nu ]\!]\|_{L^2(\partial T \setminus \partial \Omega)}^2.
\end{equation}
With these indicators at hand, we thus define the global a posteriori error estimator 
\[
\E_{\delta}^{2}(z_\T;\T):=\left(\sum_{T\in\mathscr{T}}\E_{\delta}^{2}(z_\T;T)\right)^{\frac{1}{2}}.
\]

The following result states the reliability and local efficiency of the global error estimator $\E_{\delta}$. For a proof, we refer the reader to \citep[Theorem 4.1]{ABR2006}.

\begin{theorem}
Let $z \in W_0^{1,r}(\Omega)$ and $z_{\T} \in \mathbb{V}(\T)$ be the solutions to problems \eqref{eq:dirac_eq} and \eqref{eq:discrete_dirac}, respectively.  We thus have that
\begin{equation}
\label{eq:l2estimation}
\|z-z_{\mathscr{T}}\|_{L^{2}(\Omega)} \lesssim \E_{\delta}(z_{\T};\T),
\end{equation} 
and
\begin{equation}
\label{eq:locefficiencydelta}
  \E_\delta(z_\T,T) \lesssim \| z - z_\T \|_{L^2(\mathcal{N}_\T)},
\end{equation}
where the hidden constants are independent of $z$, $z_\T$, $T$ and the cardinality of $\T$.
\end{theorem}

\begin{remark}[convexity of $\Omega$]
Assuming convexity is customary when performing an a posteriori error analysis based on duality. Indeed, the convexity of the domain $\Omega$ is imposed so that the associated dual problem exhibits suitable regularity properties which are used to show the reliability of $\E_{\delta}$ in the $L^2$--norm; see the proof of \citep[Theorem 4.1]{ABR2006} and \citep{MR2754849}.
\end{remark}


\subsection{Pointwise a posteriori error estimation}
\label{pointwise-estimator}

Since the variational inequality \eqref{eq:opt_cond}, that characterizes the optimal control, involves the duality pairing between the spaces $C_0(\Omega)$ and $\mathcal{M}(\Omega)$, it is thus imperative to consider a pointwise error estimator for the adjoint problem \eqref{eq:adjoint_pde}.

Pointwise a posteriori error estimates have been studied by several authors in the literature. To the best of our knowledge, the earliest two works that study $L^{\infty}$ residual a posteriori error estimators for a Poisson problem with a bounded forcing term are \cite{eriksson1994,nochetto1995}. The analysis of \cite{nochetto1995} was subsequently extended
to $d=3$ in \cite{DDP1999} and later improved in \cite{DG2012,Demlow2016}. The theory has also been extended to obstacle, monotone semilinear, and geometric problems \cite{camacho2015,Demlow2016,NSV2003,NSSV2006}.
A standard  requirement, in most of these works, is that the right hand side of the underlying PDE belongs to $L^{\infty}(\Omega)$. However, as it was previously mentioned, the adjoint equation \eqref{eq:adjoint_pde}, for the sparse optimal control problem, has the function $\bar{y} - y_d$ as a forcing term, which, in general is not bounded. For this reason, in what follows we will present an a posteriori error analysis in the maximum norm for a Poisson problem with an unbounded forcing term \citep{AORS2017,camacho2015,Demlow2016}. 

Let $f\in L^2(\Omega)$, and consider the following elliptic boundary value problem:
\begin{equation}\label{eq:pointwise_eq}
z \in H_0^1(\Omega): \quad (\nabla z, \nabla v)_{L^2(\Omega)} 
=
(f,v)_{L^2(\Omega)} \quad \forall v\in H_0^1(\Omega).
\end{equation}
Notice that, since we are in a convex polytope, we conclude that $z \in H^2(\Omega)$ and that this in turn implies, via Sobolev embedding, that $z \in W^{1,t}(\Omega) \cap C^{0,\kappa}(\bar\Omega)$ for some $t>d$ and $\kappa >0$. 
In view of this, it is legitimate to study the a posteriori error estimation in $L^{\infty}(\Omega)$ of problem \eqref{eq:pointwise_eq}.

We begin by defining the Galerkin approximation to problem \eqref{eq:pointwise_eq} as
\begin{equation}\label{eq:pointwise_eq_discrete}
z_\mathscr{T}\in \mathbb{V}(\mathscr{T}): \quad (\nabla z_\mathscr{T}, \nabla v_\mathscr{T})_{L^2(\Omega)} = (f,v_\mathscr{T})_{L^2(\Omega)} \quad \forall v_\mathscr{T}\in \mathbb{V}(\mathscr{T}).
\end{equation}

We thus introduce the following a posteriori local error indicators 
\begin{equation}\label{def:pointwise_estimator}
\E_{\infty}(z_\T;T):= h_T^{2-d/2} \| f \|_{L^2(T)} + 
h_T\|[\![\nabla z_\mathscr{T}\cdot \nu]\!]\|_{L^\infty(\partial T\setminus\partial \Omega)},
\end{equation}
and the error estimator $\E_\infty(z_\T;\T):=\max_{T\in \mathscr{T}}\E_\infty(z_\T;T)$.

In order to present the reliability of the global error estimator $\E_\infty$, we define
\begin{equation}\label{def:log_max}
\ell_\T:= \left|\log\bigg(\max_{T\in\mathscr{T}}\frac{1}{h_T}\bigg)\right|.
\end{equation}

The proof of the next result can be found in \citep[Lemma 4.2]{AORS2017}.

\begin{lemma}[global reliability]
Let $z \in H_0^1(\Omega)\cap L^\infty(\Omega)$ and $z_\mathscr{T}\in\mathbb{V}(\mathscr{T})$ be the solutions of \eqref{eq:pointwise_eq} and \eqref{eq:pointwise_eq_discrete}, respectively. Then 
\begin{equation}\label{eq:pointwise_error_bound}
\| z- z_\mathscr{T}\|_{L^\infty(\Omega)}
\lesssim
\ell_\T \E_\infty(z_\T;\T),
\end{equation}
where the hidden constant is independent of $f$, $z$, $z_\mathscr{T}$, the size of the elements in the mesh $\mathscr{T}$ and $\#\mathscr{T}$.
\end{lemma}

To present the local efficiency of the indicators $\E_{\infty}$, we define for any $g\in L^2(\Omega)$, and $\mathscr{M} \subset \mathscr{T}$, 
\begin{equation}\label{def:osc_term}
{\rm{osc}}_\T(g;\mathscr{M})
:=
\left(
\sum_{T\in\mathscr{M}}h_T^{2(2-d/2)}
\|g-\Pi_\T g \|_{L^2(T)}^{2}
\right)^\frac{1}{2},
\end{equation}
where $\Pi_\T$ is the $L^2$--projection operator onto piecewise linear functions over $\T$. 

The local efficiency of the indicators \eqref{def:pointwise_estimator} is as follows. For a proof we refer the reader to \citep[Lemma 4.3]{AORS2017}

\begin{lemma}[local efficiency]
Let $z \in H_0^1(\Omega)\cap L^\infty(\Omega)$ and $z_\mathscr{T}\in \mathbb{V}(\mathscr{T})$ be the solutions to problems \eqref{eq:pointwise_eq} and \eqref{eq:pointwise_eq_discrete}, respectively. Then 
\begin{equation}\label{eq:pointwise_error_efficiency}
\E_{\infty}(z_\T;T)
\lesssim
\| z-z_\mathscr{T}\|_{L^{\infty}(\mathcal{N}_T^*)}
+
{\rm{osc}}_\T(f;\mathcal{N}_T^*)
\end{equation}
for all $T \in \T$, where $\mathcal{N}_T^*$ is given by \eqref{NT_star}, and the hidden constant is independent of $f$, $z$, $z_\mathscr{T}$, the size of the elements in the mesh $\mathscr{T}$ and $\#\mathscr{T}$.
\end{lemma}


\section{A posteriori error analysis for the sparse optimal control problem}\label{global_estimator}

On the basis of the error indicators and estimators presented in Sections \ref{dirac_estimator} and \ref{pointwise-estimator}, we proceed with the design of an a posteriori error estimator for the sparse optimal control problem \eqref{eq:continuous_opc}--\eqref{eq:elliptic_pde}. The error estimator can be decomposed as the sum of two contributions:
\begin{equation}\label{def:opc_estimator}
\E_{\rm{ocp}}^2(\bar{y}_\T,\bar{p}_\T,\bar{u}_\T;\T):= \E_y^2(\bar{y}_\T,\bar{u}_\T;\T)+\E_p^2(\bar{p}_\T,\bar{y}_\T;\T),
\end{equation}
where $\T \in \mathbb{T}$ and $\bar{y}_\T,\bar{p}_\T$ and $\bar{u}_\T$ denote the discrete optimal variables that solve the discrete optimality system \eqref{eq:discrete_optimality_system}. 

Let us now describe each contribution to \eqref{def:opc_estimator} separately. First, on the basis of the results presented in Section \ref{dirac_estimator}, we define, for $T \in \T$, the local error indicators
\begin{equation}
\label{def:state_indicator}
\E_y^2(\bar{y}_\T,\bar{u}_\T;T):=  
h_T^{3}\|[\![ \nabla \bar{y}_\mathscr{T}\cdot \nu ]\!]\|_{L^2(\partial T \setminus \partial \Omega)}^2.
\end{equation}
The global error estimator $\E_y(\bar{y}_\T,\bar{u}_\T;\T)$ is thus defined by
\begin{equation}
\label{def:state_estimator}
\E_y(\bar{y}_\T,\bar{u}_\T;\T):=\left(\sum_{T\in \T} \E_y^2(\bar{y}_\T,\bar{u}_\T;T)\right)^{\frac{1}{2}}.
\end{equation}
We immediately notice that, since the optimal control $\bar u_{\T}$ is sought in the discrete space $\mathbb{U}(\T)$, it can be thus written as a linear combination of Dirac measures supported on $\{x_i\}_{i=1}^{N_\mathscr{T}}$. As a consequence, the definition of the local indicators $\E_y(\bar{y}_\T,\bar{u}_\T;T)$ does not involve the additional term $h_T^{4-d} \chi(T)$ that appears in \eqref{def:basic_estimator1}; see \cite[Remark 4.1]{ABR2006}.

The second error contribution in \eqref{def:opc_estimator} is based on the maximum--norm error estimator that we presented in Section \ref{pointwise-estimator}. Locally, it is defined by
\begin{equation}
\label{def:adjoint_indicator}
\E_p(\bar{p}_\T, \bar{y}_\T;T):= h_T^{2-d/2} \| \bar{y}_\T - y_d \|_{L^2(T)} + h_T\|[\![\nabla \bar p_\mathscr{T}\cdot \nu]\!]\|_{L^\infty(\partial T\setminus \partial \Omega)}.
\end{equation}
The global error estimator $\E_p(\bar{p}_\T,\bar{y}_\T;\T)$ is then defined by
\begin{equation}
\label{def:adjoint_estimator}
\E_p(\bar{p}_\T,\bar{y}_\T;\T):= \max_{T\in \T} \E_p(\bar{p}_\T,\bar{y}_\T;T).
\end{equation}

Since, it will be useful in the analysis that we will perform, we introduce the following auxiliary variables. First, let $(\hat{y},\hat{p})\in W_0^{1,r}(\Omega)\times H_0^1(\Omega)$ be such that
\begin{equation}\label{def:hat_variables}
\begin{array}{rcll}
(\nabla \hat{y},\nabla v)_{L^2(\Omega)} & = & \langle \bar{u}_{\mathscr{T}}, v \rangle_\Omega & \forall v\in W_0^{1,r'}(\Omega), \\
(\nabla w,\nabla\hat{p})_{L^2(\Omega)} & = & (\bar{y}_{\mathscr{T}}-y_d, w )_{L^2(\Omega)} & \forall w\in H_0^{1}(\Omega),
\end{array}
\end{equation}
where $1 \leq r < d/(d-1)$ and $r'$ denotes its conjugate exponent. We also define $\tilde{p}\in H_0^1(\Omega)$ to be the solution to
\begin{equation}\label{def:tilde_state}
(\nabla w,\nabla\tilde{p})_{L^2(\Omega)} = (\hat{y}-y_d, w )_{L^2(\Omega)} \quad \forall w\in H_0^{1}(\Omega).
\end{equation}
We notice that $(\bar{y}_\mathscr{T},\bar{p}_\mathscr{T})$ can be understood as a finite element approximation of $(\hat{y},\hat{p})$. Consequently, on the basis of the results presented in Sections \ref{dirac_estimator} and \ref{pointwise-estimator}, the a posteriori error estimators defined in \eqref{def:state_estimator} and \eqref{def:adjoint_estimator} satisfy the following reliability properties:
\begin{align}
\label{eq:state_and_adjoint_bounds}
\|\hat{y}-\bar{y}_\mathscr{T}\|_{L^2(\Omega)}
\lesssim 
\E_y(\bar{y}_\T,\bar{u}_\T;\T),
\qquad
\|\hat{p}-\bar{p}_\mathscr{T}\|_{L^\infty(\Omega)}
\lesssim
\ell_\T \E_p(\bar{p}_\T,\bar{y}_\T;\T) .
\end{align}
We thus have all the ingredients at hand to develop our a posteriori error analysis for the sparse optimal control problem \eqref{eq:continuous_opc}--\eqref{eq:elliptic_pde}.


\subsection{A posteriori error estimator: reliability}\label{reliability}

\begin{theorem}[global reliability] Let $(\bar{y},\bar{p},\bar{u})\in W_0^{1,r}(\Omega) \times H_0^1(\Omega) \times \mathcal{M}(\Omega)$, with $2d/(d+2) \leq r < d/(d-1)$, be the solution to \eqref{eq:optimality_system}, and let $(\bar{y}_\mathscr{T},\bar{p}_\mathscr{T},\bar{u}_\mathscr{T})\in \mathbb{V}(\mathscr{T})\times \mathbb{V}(\mathscr{T})\times \mathbb{U}(\mathscr{T})$ be its numerical approximation obtained as the solution to the discrete optimality system \eqref{eq:discrete_optimality_system}.  Then
\begin{multline}
\|\bar{y}-\bar{y}_\mathscr{T}\|_{L^2(\Omega)}^2 
+ 
\|\bar{p}-\bar{p}_\mathscr{T}\|_{L^\infty(\Omega)}^2
+ 
\|\bar{u}-\bar{u}_\mathscr{T}\|_{H^{-2}(\Omega)}^2  \\
 \lesssim  
\E_y^2(\bar{y}_\T,\bar{u}_\T;\T)+\ell^2_\T\E_p^2(\bar{p}_\T,\bar{y}_\T;\T)+\ell_\T\E_p(\bar{p}_\T,\bar{y}_\T;\T), 
\label{eq:estimador_bound}
\end{multline}
where $\ell_\T$ is defined in \eqref{def:log_max}, and the hidden constant is independent of the continuous and discrete optimal variables, the size of the elements of the mesh $\mathscr{T}$ and its cardinality $\# \mathscr{T}$.
\end{theorem}
\begin{proof}
We proceed in four steps.

\noindent \underline{\emph{Step 1}}. The objective of this step is to bound the error $\|\bar{y}-\bar{y}_\mathscr{T}\|_{L^2(\Omega)}$. To accomplish this task, we invoke the auxiliary state variable $\hat{y}$, defined as the solution to \eqref{def:hat_variables}, and apply the triangle inequality to arrive at the estimate
\begin{equation}
\label{eq:state_estimate_1}
\|\bar{y}-\bar{y}_\mathscr{T}\|_{L^2(\Omega)}^2
\lesssim
\|\bar{y}-\hat{y}\|_{L^2(\Omega)}^2 + \E_y^2(\bar{y}_\T,\bar{u}_\T;\T),
\end{equation}
where we have also used \eqref{eq:state_and_adjoint_bounds}. We now focus on controlling the term $\|\bar{y}-\hat{y}\|_{L^2(\Omega)}^2$. Set 
$u=\bar{u}_\mathscr{T}$ in \eqref{eq:opt_cond} and $u=\bar{u}$ in \eqref{eq:discrete_opt_cond}, and obtain that
\begin{align*}
-\langle \bar{u}_\mathscr{T} - \bar{u}, \bar{p}\rangle_\Omega + \alpha \|\bar{u}\|_{\mathcal{M}(\Omega)}
\leq & \;
\alpha \| \bar{u}_\mathscr{T} \|_{\mathcal{M}(\Omega)}, \\
-\langle \bar{u} - \bar{u}_\mathscr{T}, \bar{p}_\mathscr{T}\rangle_\Omega + \alpha \|\bar{u}_\mathscr{T}\|_{\mathcal{M}(\Omega)}
\leq & \;
\alpha \| \bar{u} \|_{\mathcal{M}(\Omega)}.
\end{align*}
Adding the previous inequalities, we thus obtain that
\begin{equation}\label{eq:equation_1}
\langle \bar{u}-\bar{u}_\mathscr{T},\bar{p}-\tilde{p}\rangle_\Omega + \langle \bar{u}-\bar{u}_\mathscr{T},\tilde{p} - \bar{p}_\mathscr{T}\rangle_\Omega \leq 0,
\end{equation}
where we have invoked the auxiliary adjoint state $\tilde{p}$, defined in \eqref{def:tilde_state}.

Let us concentrate now on the term $\langle \bar{u}-\bar{u}_\mathscr{T},\bar{p}-\tilde{p}\rangle_\Omega$. First, we notice that the functions $\bar{y}-\hat{y}$ and $\bar{p}-\tilde{p}$ satisfy
\begin{align}
\label{eq:equation_2}
\bar{y}-\hat{y} \in W_0^{1,r}(\Omega): \quad (\nabla (\bar{y}-\hat{y}),\nabla v)_{L^2(\Omega)} =  \langle\bar{u}-\bar{u}_\mathscr{T},v\rangle_\Omega  \quad  \forall v\in W_0^{1,r'}(\Omega),\\
\label{eq:equation_22}
\bar{p}-\tilde{p} \in H_0^1(\Omega): \quad
(\nabla w,\nabla (\bar{p}-\tilde{p}))_{L^2(\Omega)} =  (\bar{y}-\hat{y},w)_{L^2(\Omega)} \quad \forall w\in H_0^1(\Omega),
\end{align}
where $2d/(d+2) \leq r < d/(d-1)$ and $r'>d$ is its conjugate exponent. Notice now that, since $\bar{y}-\hat{y} \in L^2(\Omega)$, there exists $r'>d$ for which $\bar{p}-\tilde{p} \in W^{1,r'}(\Omega)$. Consequently, we are able to set $v =\bar{p}-\tilde{p}$ in \eqref{eq:equation_2}. This yields
\[
(\nabla (\bar{y}-\hat{y}),\nabla (\bar{p}-\tilde{p}))_{L^2(\Omega)} =  \langle\bar{u}-\bar{u}_\mathscr{T},\bar{p}-\tilde{p}\rangle_\Omega.
\]
Following a similar reasoning, we would like to set $w = \bar{y}-\hat{y}$ in \eqref{eq:equation_22}. However, $ \bar{y}-\hat{y} \notin H_0^1(\Omega)$. Exploiting the fact that $\bar{y}-\hat{y} \in W_0^{1,r}(\Omega)$ with $2d/(d+2) \leq r < d/(d-1)$ and that there exits $r' > d$ such that $\bar{p}-\tilde{p} \in W^{1,r'}(\Omega)$, a density argument allows us to conclude that
\[
(\nabla ( \bar{y}-\hat{y}),\nabla (\bar{p}-\tilde{p}))_{L^2(\Omega)}  =  (\bar{y}-\hat{y}, \bar{y}-\hat{y})_{L^2(\Omega)},
\]
and thus that
\begin{equation}
\label{eq:duality}
\langle \bar{u}-\bar{u}_\mathscr{T},\bar{p}-\tilde{p}\rangle_\Omega 
= 
\|\bar{y}-\hat{y}\|_{L^2(\Omega)}^2.
\end{equation}
Replacing the previous term in \eqref{eq:equation_1}, and inserting the auxiliary variable $\hat{p}$ defined as the solution to \eqref{def:hat_variables}, we obtain that
\begin{equation}\label{eq:equation_3}
\|\bar{y}-\hat{y}\|_{L^2(\Omega)}^2 
\leq
\langle \bar{u}_\mathscr{T} - \bar{u},\tilde{p} - \hat{p}\rangle_\Omega
+
\langle \bar{u}_\mathscr{T} - \bar{u},\hat{p} - \bar{p}_\mathscr{T}
\rangle_\Omega.
\end{equation}

We now apply Theorem \ref{discrete_existence_sol} to conclude that the norms $\|\bar{u}\|_{\mathcal{M}(\Omega)}$ and $\|\bar{u}_\mathscr{T}\|_{\mathcal{M}(\Omega)}$ are bounded. This, combined with the estimate \eqref{eq:state_and_adjoint_bounds}, reveals that
\begin{equation}\label{eq:equation_4}
\langle \bar{u}_\mathscr{T} - \bar{u},\hat{p} - \bar{p}_\mathscr{T}
\rangle_\Omega 
\lesssim \ell_\T \E_p(\bar{p}_\T,\bar{y}_\T;\T).
\end{equation}

To estimate the first term on the right hand side of \eqref{eq:equation_3}, we notice that $\hat{y}-\bar{y}$ and $\tilde{p}-\hat{p}$ satisfy
\begin{align}
\label{eq:equation_middle}
\hat{y}-\bar{y} \in W_0^{1,r}(\Omega): \quad (\nabla (\hat{y}-\bar{y}),\nabla v)_{L^2(\Omega)} =  \langle\bar{u}_\mathscr{T}-\bar{u},v\rangle_\Omega \quad  \forall v\in W_0^{1,r'}(\Omega),\\
\label{eq:equation_middle2}
\tilde{p}-\hat{p} \in H_0^1(\Omega): \quad
(\nabla w,\nabla (\tilde{p}-\hat{p}))_{L^2(\Omega)} =  (\hat{y}-\bar{y}_\mathscr{T},w)_{L^2(\Omega)} \quad  \forall w\in H_0^1(\Omega).
\end{align}
Similar density arguments to those used to obtain \eqref{eq:duality}, allow us to set $v=\tilde{p}-\hat{p}$ and $w=\hat{y}-\bar{y}$. This yields
\begin{equation}\label{eq:equation_5}
\langle \bar{u}_\mathscr{T} - \bar{u},\tilde{p} - \hat{p}\rangle_\Omega = (\hat{y}-\bar{y}_\mathscr{T},\hat{y}-\bar{y}).
\end{equation}
Replacing \eqref{eq:equation_4} and the obtained result into \eqref{eq:equation_3} we thus obtain that 
\begin{equation}\label{eq:equation_6}
\|\bar{y}-\hat{y}\|_{L^2(\Omega)}^2
\lesssim
\|\hat{y}-\bar{y}_\mathscr{T} \|_{L^2(\Omega)}^2
+
(\hat{y}-\bar{y}_\mathscr{T},\bar{y}_\mathscr{T}-\bar{y})_{L^2(\Omega)} + \ell_\T \E_p(\bar{p}_\T,\bar{y}_\T;\T).
\end{equation}
Now, on the basis of \eqref{eq:state_and_adjoint_bounds}, an application of Young's inequality allows us to arrive at the estimate
\begin{equation}\label{eq:state_estimate_2}
\|\bar{y}-\hat{y}\|_{L^2(\Omega)}^2
\lesssim \E_y^2(\bar{y}_\T,\bar{u}_\T;\T)+ \ell_\T \E_p(\bar{p}_\T,\bar{y}_\T;\T) +\frac{1}{4}\|\bar{y}-\bar{y}_\mathscr{T}\|_{L^2(\Omega)}^2.
\end{equation}
Finally, replacing \eqref{eq:state_estimate_2} into \eqref{eq:state_estimate_1}, we obtain the estimate
\begin{equation}\label{eq:state_estimate_3}
\|\bar{y}-\bar{y}_\mathscr{T}\|_{L^2(\Omega)}^2
\lesssim
\E_y^2(\bar{y}_\T,\bar{u}_\T;\T) + \ell_\T \E_p(\bar{p}_\T,\bar{y}_\T;\T).
\end{equation}

\noindent \underline{\emph{Step 2.}} The goal of this step is to estimate the term $\|\bar{u}-\bar{u}_\mathscr{T}\|_{H^{-2}(\Omega)}$. To accomplish this task, we follow the arguments elaborated in \cite[Corollary 4.5]{PV2013} that guarantee 
\begin{equation}\label{eq:control_estimate_1}
\|\bar{u}-\bar{u}_\mathscr{T}\|_{H^{-2}(\Omega)}^2
\lesssim
\| \bar{y}-\bar{y}_\mathscr{T}\|_{L^2(\Omega)}^2
+
\| \bar{y}_\mathscr{T} - \hat{y} \|_{L^2(\Omega)}^2,
\end{equation}
where $\hat y$ is defined as the solution to \eqref{def:hat_variables}. We now invoke, once again, the estimate \eqref{eq:state_and_adjoint_bounds} together with \eqref{eq:state_estimate_3} to conclude that
\begin{equation}\label{eq:control_estimate_2}
\|\bar{u}-\bar{u}_\mathscr{T}\|_{H^{-2}(\Omega)}^2
\lesssim
\E_y^2(\bar{y}_\T,\bar{u}_\T;\T) + \ell_\T \E_p(\bar{p}_\T,\bar{y}_\T;\T).
\end{equation}

\noindent \underline{\emph{Step 3.}} In this step we bound the error $\|\bar{p}-\bar{p}_\mathscr{T}\|_{L^\infty(\Omega)}$.
To accomplish this task, we use the triangle inequality and then the estimate \eqref{eq:state_and_adjoint_bounds} to obtain that
\begin{equation}\label{eq:adjoint_estimate_1}
\|\bar{p}-\bar{p}_\mathscr{T}\|_{L^\infty(\Omega)}^2 \lesssim \|\bar{p}-\hat{p}\|_{L^\infty(\Omega)}^2 +
\ell_\T^2 \E_p^2(\bar{p}_\T,\bar{y}_\T;\T).
\end{equation}
To estimate the term $\|\bar{p}-\hat{p}\|_{L^\infty(\Omega)}^2$, we 
observe that $\bar{p}-\hat{p} \in H^2(\Omega)\hookrightarrow C(\bar{\Omega})$ so that this, together with \eqref{eq:state_estimate_3} gives
\begin{equation}\label{eq:adjoint_estimate_2}
\|\bar{p}-\hat{p}\|_{L^\infty(\Omega)}^2
\lesssim
\| \bar{y}-\bar{y}_\mathscr{T}\|_{L^2(\Omega)}^2
\lesssim
\E_y^2(\bar{y}_\T,\bar{u}_\T;\T) + \ell_\T \E_p(\bar{p}_\T,\bar{y}_\T;\T). 
\end{equation}
Finally, inserting \eqref{eq:adjoint_estimate_2} into \eqref{eq:adjoint_estimate_1}, we arrive at the estimate
\begin{equation}\label{eq:adjoint_estimate_3}
\|\bar{p}-\bar{p}_\mathscr{T}\|_{L^\infty(\Omega)}^2
\lesssim
\E_y^2(\bar{y}_\T,\bar{u}_\T;\T) + \ell^2_\T \E_p^2(\bar{p}_\T,\bar{y}_\T;\T)+\ell_\T \E_p(\bar{p}_\T,\bar{y}_\T;\T).
\end{equation}

\noindent \underline{\emph{Step 4.}} The desired estimate \eqref{eq:estimador_bound} follows upon gathering the estimates \eqref{eq:state_estimate_3}, \eqref{eq:control_estimate_2} and \eqref{eq:adjoint_estimate_3}.
\end{proof}


\subsection{A posteriori error estimator: efficiency}\label{efficiency}
In this section we analyze the efficiency properties of the local a posteriori error indicator 
\begin{equation}\label{def:opc_estimator_local}
\E_{\rm{ocp}}^2(\bar{y}_\T,\bar{p}_\T,\bar{u}_\T;T)= \E_y^2(\bar{y}_\T,\bar{u}_\T;T)+\E_p^2(\bar{p}_\T,\bar{y}_\T;T).
\end{equation}
To accomplish this task, we study each of its contributions separately. We start with the indicator $\E_y(\bar{y}_\T,\bar{u}_\T; T)$ defined by \eqref{def:state_indicator}. Before embarking ourselves with the efficiency analysis of $\E_y(\bar{y}_\T,\bar{u}_\T; T)$, we introduce the following notation: for an edge, triangle or tetrahedron $G$, let $\mathcal{V}(G)$ be the set of vertices of $G$.

Let $T \in \T$ and $S \in \Sides_T$. Recall that $\mathcal{N}_{S}$ denotes the patch composed by the two elements $T$ and $T'$ sharing $S$. We introduce the following edge bubble function 
\begin{equation}\label{def:bubble_functions}
\psi_{S}|_{\mathcal{N}_{S}}=d^{4d}
\left(\prod_{\texttt{v}\in\mathcal{V}(S)} \phi_{\texttt{v}}^{T} \phi_{\texttt{v}}^{T'}\right)^{2},
\end{equation}
where, for $\texttt{v} \in \mathcal{V}(S)$, $\phi_{\texttt{v}}^T$ and $\phi_{\texttt{v}}^{T'}$ denote the barycentric coordinates of $T$ and $T'$, respectively, which are understood as functions over $\mathcal{N}_{S}$. The following properties of the bubble function $\psi_S$ follow immediately: $\psi_{S} \in \mathbb{P}_{4d}(\mathcal{N}_{S})$, $\psi_{S} \in C^2(\mathcal{N}_{S})$, and $\psi_{S} = 0$ on $\partial \mathcal{N}_{S}$. In addition, we have that
\begin{equation}
\label{eq:delta_gamma}
 \nabla \psi_{S} = 0 \textrm{ on } \partial \mathcal{N}_{S}, \quad 
 [\![\nabla \psi_S\cdot\nu]\!]= 0 \textrm{ on } S.
\end{equation}
With all these ingredients at hand, we are ready to prove the local efficiency of $\E_y^2(\bar{y}_\T,\bar{u}_\T; T)$.

\begin{lemma}[local efficiency of $\E_y$]\label{efficiency_est_y}
Let $(\bar{y},\bar{p},\bar{u})\in W_0^{1,r}(\Omega) \times H_0^1(\Omega)\times \mathcal{M}(\Omega)$, with $2d/(d+2)\leq r < d/(d-1)$, be the solution to \eqref{eq:optimality_system}, and let $(\bar{y}_\mathscr{T},\bar{p}_\mathscr{T},\bar{u}_\mathscr{T})\in \mathbb{V}(\mathscr{T})\times \mathbb{V}(\mathscr{T})\times \mathbb{U}(\mathscr{T})$ be its numerical approximation obtained as the solution to the discrete optimality system \eqref{eq:discrete_optimality_system}. Then, for $T\in\mathscr{T}$, the local error indicator $\E_y$, defined as in \eqref{def:state_indicator}, satisfies that
\begin{equation}\label{eq:state_efficiency}
\E_y^2(\bar{y}_\T,\bar{u}_\T; T)
\lesssim
\|\bar{y}-\bar{y}_\mathscr{T}\|_{L^2(\mathcal{N}_T^*)}^2
+
\sum_{S \in \Sides_T} \|\bar{u}-\bar{u}_\mathscr{T}\|_{H^{-2}(\mathcal{N}_S)}^2,
\end{equation}
where $\mathcal{N}_T^*$ is defined as in \eqref{NT_star} and the hidden constant is independent of the continuous and discrete optimal variables,
the size of the elements in the mesh $\mathscr{T}$, and $\#\mathscr{T}$.
\end{lemma}
\begin{proof} Let $v \in W_0^{1,r'}(\Omega)$, with $r'>d$, be such that $v|_T\in C^2(T)$ for all $T\in \mathscr{T}$. Consider $v$ as a test function in the state equation \eqref{eq:weak_pde} and apply
integration by parts to arrive at
\begin{equation}\label{eq:state_ibp_1}
\int_\Omega \nabla (\bar{y} - \bar{y}_\T) \nabla v = \sum_{T\in \T} \langle \bar{u}, v \rangle_T  + \sum_{S\in\Sides}\int_S [\![\nabla \bar{y}_\mathscr{T}\cdot \nu ]\!] v.
\end{equation}
Since we have that $v\in C^2(T)$ on each $T\in \T$, we can integrate by parts, again, to conclude that
\begin{equation}\label{eq:state_ibp_2}
\int_\Omega \nabla (\bar{y} - \bar{y}_\T) \nabla v =- \sum_{S\in\Sides}\int_S [\![\nabla v\cdot \nu ]\!] (\bar{y}-\bar{y}_\mathscr{T}) - \sum_{T\in \T}\int_T(\bar{y}-\bar{y}_\mathscr{T}) \Delta v.
\end{equation}
As a conclusion, from \eqref{eq:state_ibp_1} and \eqref{eq:state_ibp_2}, we arrive at the following identity:
\begin{multline}
\label{eq:state_ibp_3}
\sum_{T\in \T} \left(\langle \bar{u}-\bar{u}_\T, v \rangle_T + \langle \bar{u}_\T, v \rangle_T \right) + \sum_{S\in\Sides}\int_S [\![\nabla \bar{y}_\mathscr{T}\cdot \nu ]\!] v
\\ 
=  - \sum_{S\in\Sides}\int_S [\![\nabla v\cdot \nu ]\!] (\bar{y}-\bar{y}_\mathscr{T}) - \sum_{T\in \T}\int_T(\bar{y}-\bar{y}_\mathscr{T}) \Delta v,  
\end{multline}
which holds for every $v\in W_0^{1,r'}(\Omega)$, with $r'>d$, and is such that $v|_T\in C^2(T)$ for all $T\in\T$.

Let $S\in \Sides$, and set $v=\beta_S=[\![\nabla \bar{y}_\mathscr{T}\cdot\nu]\!]\psi_S$ in \eqref{eq:state_ibp_3}. Notice that we have that $\beta_S \in H^2_0(\mathcal{N}_S)$ and, moreover, $[\![\nabla \beta_S \cdot \nu]\!] = 0$ on $S$. In addition, since
$\bar{u}_\mathscr{T} \in \mathbb{U}(\T)$, we can conclude for every $T\in \mathcal{N}_S$ that
\[
 \langle \bar{u}_\T, \beta_S \rangle_T = \sum_{\texttt{v} \in \mathcal{V}(T)} u_\texttt{v}  \langle \delta_{\texttt{v}}, \beta_S \rangle_T = 0.
\]
Consequently, a simple application of the Cauchy--Schwarz inequality reveals that
\begin{align}\label{eq:estimate_1}
& \int_S [\![\nabla \bar{y}_\mathscr{T}\cdot\nu]\!]^2\psi_S
 =
-\sum_{T'\in \mathcal{N}_S} 
\bigg( \int_{T'}(\bar{y}-\bar{y}_\mathscr{T})\Delta \beta_S + \langle \bar{u}-\bar{u}_\mathscr{T},\beta_S\rangle_{T'}\bigg) \\ \nonumber
& \leq \sum_{T'\in \mathcal{N}_S} \bigg( \| \bar{y} - \bar{y}_\mathscr{T} \|_{L^2(T')} \| \Delta \beta_S \|_{L^2(T')} \bigg) 
+
\| \bar{u} - \bar{u}_\mathscr{T} \|_{H^{-2}(\mathcal{N}_S)} \| \beta_S \|_{H^{2}(\mathcal{N}_S)} . \nonumber
\end{align}

Now, since $[\![\nabla \bar{y}_\mathscr{T}\cdot \nu]\!]\in \mathbb{R}$, we have that $\Delta \beta_S = [\![\nabla \bar{y}_\mathscr{T}\cdot\nu]\!] \Delta\psi_S$. Standard arguments allow us to obtain the following bound 
\begin{align*}
\| \Delta \beta_S \|_{L^2(T)}^2 & \lesssim [\![\nabla \bar{y}_\mathscr{T}\cdot\nu]\!]^2  \| \Delta\psi_S \|_{L^2(T)}^2
\\
& \lesssim |T| |S|^{-1} h_T^{-2}  [\![\nabla \bar{y}_\mathscr{T}\cdot\nu]\!] \|_{L^2(S)}^2
\lesssim
h_T^{-3}  \|[\![\nabla \bar{y}_\mathscr{T}\cdot\nu]\!]\|_{L^2(S)}^2.
\end{align*} 
With these estimates at hand, we thus use standard bubble functions arguments, the Poincaré inequality, and the shape regularity property of the family $\{ \mathscr{T} \}$ to arrive at

\begin{equation}\label{eq:estimate_2}
h_T^3\| [\![\nabla \bar{y}_\mathscr{T}\cdot\nu ]\!]\|_{L^2(S)}^2
\lesssim
\sum_{T'\in\mathcal{N}_S} 
\bigg\{
\|\bar{y}-\bar{y}_\mathscr{T}\|_{L^2(T')}^2 \bigg\}
+
\|\bar{u}-\bar{u}_\mathscr{T}\|_{H^{-2}(\mathcal{N}_S)}^2,
\end{equation}
which concludes our proof.
\end{proof}

We now continue with the study of the local efficiency properties of the indicators $\E_p(\bar{p}_\T, \bar{y}_\T;T)$ defined by \eqref{def:adjoint_indicator}.

\begin{lemma}[local efficiency of $\E_p$]\label{efficiency_est_p}
Let $(\bar{y},\bar{p},\bar{u})\in W_0^{1,r}(\Omega) \times H_0^1(\Omega) \times \mathcal{M}(\Omega)$, where $2d/(d+2)\leq r < d/(d-1)$, be the solution to the optimality system \eqref{eq:optimality_system}, and $(\bar{y}_\mathscr{T},\bar{p}_\mathscr{T},\bar{u}_\mathscr{T})\in \mathbb{V}(\mathscr{T})\times \mathbb{V}(\mathscr{T})\times \mathbb{U}(\mathscr{T})$ be its numerical approximation obtained as the solution to the discrete optimality system \eqref{eq:discrete_optimality_system}. Then, for $T\in\mathscr{T}$, the local error indicator $\E_p(\bar{p}_\T, \bar{y}_\T;T)$, defined in \eqref{def:adjoint_indicator}, satisfies that 
\begin{equation}\label{eq:adjoint_efficiency}
\E_p(\bar{p}_\T, \bar{y}_\T;T)
\lesssim
\|\bar{p}-\bar{p}_\mathscr{T}\|_{L^\infty(\mathcal{N}_T^*)}
+
h_T^{2-d/2}\|\bar{y}-\bar{y}_\mathscr{T}\|_{L^{2}(\mathcal{N}_T^*)}
+
{\rm{osc}}_\T(y_d; \mathcal{N}_T^*),
\end{equation}
where $\mathcal{N}^*_T$ is defined in \eqref{NT_star} and the hidden constant is independent of the optimal variables, their approximations, the size of the elements in the mesh $\mathscr{T}$, and its cardinality $\#\mathscr{T}$.
\end{lemma}
\begin{proof} We follow closely the arguments elaborated in \cite[Lemma 5.4]{AORS2017}. Let $w\in H_0^1(\Omega)$ be such that $w|_T\in C^2(\Omega)$ for $T \in \mathscr{T}$. Consider $w$ as a test function in the adjoint equation \eqref{eq:adjoint_pde}.
An argument based on integration by parts reveals that
\begin{align}\label{eq:adjoint_ibp_1}
\begin{split}
\int_\Omega \nabla w \nabla (\bar{p} - \bar{p}_\T) = 
\sum_{T\in \T} \int_T(\bar{y}-y_d) w + \sum_{S\in\Sides}\int_S [\![\nabla \bar{p}_\mathscr{T}\cdot \nu ]\!] w.
\end{split}
\end{align}
In view of the regularity properties of the function $w$, we can, again, integrate by parts to obtain that
\begin{equation}\label{eq:adjoint_ibp_2}
\int_\Omega  \nabla w \nabla (\bar{p} - \bar{p}_\T) =- \sum_{S\in\Sides}\int_S [\![\nabla w\cdot \nu ]\!] (\bar{p}-\bar{p}_\mathscr{T}) - \sum_{T\in \T}\int_T(\bar{p}-\bar{p}_\mathscr{T}) \Delta w.
\end{equation}
As a conclusion, we have thus obtained the following identity
\begin{multline}\label{eq:adjoint_ibp_3}
 \sum_{T\in \T} 
 \int_T(\bar{y}-y_d)w 
 +\sum_{S\in\Sides}\int_S [\![\nabla \bar{p}_\mathscr{T}\cdot \nu ]\!] w \\
 = - \sum_{S\in\Sides}\int_S [\![\nabla w\cdot \nu ]\!] (\bar{p}-\bar{p}_\mathscr{T}) - \sum_{T\in \T}\int_T(\bar{p}-\bar{p}_\mathscr{T}) \Delta w,
\end{multline}
that holds for every $w \in H_0^1(\Omega)$ such that $w|_T\in C^2(T)$ for all $T\in\T$.
With this error equation at hand, we proceed to derive the local efficiency properties of the error indicator $\E_p(\bar{p}_\T, \bar{y}_\T;T)$: the analysis involves two steps.

\noindent \underline{\emph{Step 1.}} In this step we bound the term $h_T^{2-d/2} \| \bar{y}_\T - y_d \|_{L^2(T)}$ in \eqref{def:adjoint_indicator} for $T\in \T$. We begin with a simple application of the triangle inequality:
\begin{equation*}
h_T^{2-d/2} \| \bar{y}_\T - y_d \|_{L^2(T)}
\leq
h_T^{2-d/2} \| \bar{y}_\T - \Pi_\T y_d \|_{L^2(T)} + h_T^{2-d/2} \| \Pi_\T y_d - y_d \|_{L^2(T)},
\end{equation*}
where $\Pi_\T$ denotes the $L^2$--projection operator onto piecewise linear functions over $\T$.
Now we set $w=\zeta_T:=(\bar{y}_\mathscr{T}-\Pi_\T y_d)\varphi_T^2$ in \eqref{eq:adjoint_ibp_3}, where $\varphi_T$ is the standard bubble function over $T$ \citep{ver1989,ver2013}. This yields
\begin{multline}\label{eq:estimate_1_adjoint}
\int_T (\bar{y}_\mathscr{T}-\Pi_\T y_d)\zeta_T
= \int_T \left[ (\bar{y} -y_d) - (\bar y - \bar y_{\T}) + (y_d - \Pi_{\T} y_d) \right] \zeta_T
\\ 
=  - \int_T(\bar{p}-\bar{p}_\mathscr{T}) \Delta \zeta_T
-
\int_T(\bar{y}-\bar{y}_\mathscr{T})\zeta_T + \int_T(y_d - \Pi_\T y_d)\zeta_T =: \rm{I + II + III}.
\end{multline}

To compute $\Delta \zeta_T$ on $T$, we first notice that $\Delta (\bar{y}_\mathscr{T}-\Pi_\T y_d)|_T=0$. A basic computation thus reveals that
\begin{equation*}
\Delta \zeta_T = 4 \nabla (\bar{y}_\mathscr{T}-\Pi_\T y_d) \cdot \nabla	\varphi_T \varphi_T
+ 
2(\bar{y}_\mathscr{T}-\Pi_\T y_d)(\varphi_T \Delta \varphi_T+\nabla \varphi_T \cdot \nabla \varphi_T).
\end{equation*}

The control of the term $\textrm{I}$ follows from the the previous identity, the properties of the bubble function $\varphi_T$ and an inverse inequality. In fact, we have that
\begin{multline}
\label{eq:estimate_2_adjoint}
|\textrm{I}| 
\lesssim
\left( h_T^{\frac{d}{2}-1}\|\nabla (\bar{y}_\mathscr{T}-\Pi_\T y_d)\|_{L^2(T)}
+
h_T^{\frac{d}{2}-2}\|\bar{y}_\mathscr{T}-\Pi_\T y_d\|_{L^2(T)}\right)
\\
\cdot \|\bar{p}-\bar{p}_\mathscr{T}\|_{L^\infty(T)}
\lesssim 
h_T^{d/2-2}\|\bar{y}_\mathscr{T}-\Pi_\T y_d\|_{L^2(T)}
\|\bar{p}-\bar{p}_\mathscr{T}\|_{L^\infty(T)}.
\end{multline}

The terms $\rm{II}$ and $\rm{III}$ in \eqref{eq:estimate_1_adjoint} are bounded in view of standard properties of the bubble function $\varphi_T$: we have that
\begin{equation}
|\mathrm{II}|
\lesssim
\|\bar{y}-\bar{y}_\mathscr{T}\|_{L^2(T)}
\|\bar{y}_\mathscr{T}-\Pi_\T y_d\|_{L^2(T)}
\label{eq:estimate_3_adjoint}
\end{equation}
and that
\begin{equation}
|\mathrm{III}| 
\lesssim 
\| y_d - \Pi_\T y_d\|_{L^2(T)}
\|\bar{y}_\mathscr{T}-\Pi_\T y_d\|_{L^2(T)}.
\label{eq:estimate_4_adjoint}
\end{equation}

Finally, since $\|\bar{y}_\mathscr{T}- \Pi_\T y_d\|^2_{L^2(T)} \lesssim \int_T (\bar{y}_\mathscr{T}- \Pi_\T y_d) \zeta_T$, the identity \eqref{eq:estimate_1_adjoint}, combined with the estimates \eqref{eq:estimate_2_adjoint}, \eqref{eq:estimate_3_adjoint}  and \eqref{eq:estimate_4_adjoint}, allow us to conclude that
\begin{multline}\label{eq:residual_bound}
h_T^{2-\frac{d}{2}}\|\bar{y}_\mathscr{T}-y_d\|_{L^2(T)}
\lesssim  \|\bar{p}-\bar{p}_\mathscr{T}\|_{L^\infty(T)}\\
+
h_T^{2-\frac{d}{2}}\|\bar{y}-\bar{y}_\mathscr{T}\|_{L^2(T)} + {\rm{osc}}_\T(y_d;T),
\end{multline}
where ${\rm{osc}}_\T(y_d;T)$ is defined by \eqref{def:osc_term}.

\noindent \underline{\emph{Step 2.}} Let $T\in\mathscr{T}$ and $S\in \Sides_T$. The objetive of this step is control the term $h_T\|[\![ \nabla \bar{p}_\mathscr{T}\cdot \nu]\!]\|_{L^\infty(\partial T\setminus \partial \Omega)}$ in \eqref{def:adjoint_indicator}. To accomplish this task, 
we invoke the standard bubble function over $S$ which we denote $\varphi_S$. Recall that, according to \citep{ver1989,ver2013} this function satisfies
\begin{equation*}
|\nabla^l \varphi_S|\approx h_S^{-l}, \quad l=0,1,2,
\qquad
|S|\|[\![\nabla \bar{p}_\mathscr{T} \cdot \nu]\!] \|_{L^\infty(S)}
\lesssim
\left| \int_S [\![\nabla \bar{p}_\mathscr{T} \cdot \nu]\!] \varphi_S  \right|.
\end{equation*}
To bound the term on the right--hand side of the previous inequality, we set $w=\varphi_S$ in \eqref{eq:adjoint_ibp_3} and obtain that
\begin{multline*}
\left| \int_S [\![\nabla \bar{p}_\mathscr{T} \cdot \nu]\!] \varphi_S  \right|  
\leq
\sum_{T'\in\mathcal{N}_S}\int_{T'}|\bar{y}-y_d|\varphi_S
+
\sum_{T'\in\mathcal{N}_S}\int_{T'}|\bar{p}-\bar{p}_\mathscr{T}||\Delta \varphi_S|
\\
+ \sum_{T'\in\mathcal{N}_S}\sum_{S'\in\Sides_{T'}}
\int_{S'}|\bar{p}-\bar{p}_\mathscr{T}|\left|[\![\nabla \varphi_S \cdot \nu]\!]\right| \\ \nonumber
\lesssim
\sum_{T'\in\mathcal{N}_S}|T'|^{\frac{1}{2}}
(\|\bar{y}-\bar{y}_\mathscr{T}\|_{L^2(T')}
+
\|\bar{y}_\mathscr{T}-y_d\|_{L^2(T')}) 
\\
+\sum_{T'\in \mathcal{N}_{S}}
\bigg(
h_S^{-2}|T'|
+
h_S^{-1}\sum_{{S'}\in \Sides_{T'}}|S'|
\bigg)
\|\bar{p}-\bar{p}_\mathscr{T}\|_{L^\infty(T')}.
\end{multline*} 
Upon multiplying this inequality by $h_T|S|^{-1}$, using the shape regularity of the mesh $\T$, in addition with \eqref{eq:residual_bound}, yields 
\begin{multline}\label{eq:jump_bound}
h_T\|[\![\nabla \bar{p}_\mathscr{T}\cdot \nu]\!]\|_{L^\infty(S)}
\lesssim \|\bar{p}-\bar{p}_\mathscr{T}\|_{L^\infty(\mathcal{N}_S)}
\\
+
h_T^{2-d/2}\|\bar{y}-\bar{y}_\mathscr{T}\|_{L^2(\mathcal{N}_S)}
+
{\rm{osc}}_\T(y_d;\mathcal{N}_S).
\end{multline}
Combining \eqref{eq:residual_bound} and \eqref{eq:jump_bound}, we arrive at \eqref{eq:adjoint_efficiency}.
\end{proof}

The results of Lemmas \ref{efficiency_est_y} and \ref{efficiency_est_p} immediately yield the following result.

\begin{theorem}[local efficiency of $\E_{\rm{ocp}}$]\label{efficiency_est_ocp}
Let $(\bar{y},\bar{p},\bar{u})\in W_0^{1,r}(\Omega) \times H_0^1(\Omega) \times \mathcal{M}(\Omega)$, with $2d/(d+2)\leq r < d/(d-1)$, be the solution to the optimality system \eqref{eq:optimality_system}, and $(\bar{y}_\mathscr{T},\bar{p}_\mathscr{T},\bar{u}_\mathscr{T})\in \mathbb{V}(\mathscr{T})\times \mathbb{V}(\mathscr{T})\times \mathbb{U}(\mathscr{T})$ be its numerical approximation obtained as the solution to the discrete optimality system \eqref{eq:discrete_optimality_system}. Then, for $T\in\mathscr{T}$, it holds
\begin{multline*}
\E_y^2(\bar{y}_\T, \bar{u}_\T;T)  + \E_p^2(\bar{p}_\T, \bar{y}_\T;T) 
\lesssim 
\|\bar{p}-\bar{p}_\mathscr{T}\|_{L^\infty(\mathcal{N}_T^*)}^{2}
+
(1+h_T^{4-d})\|\bar{y}-\bar{y}_\mathscr{T}\|_{L^{2}(\mathcal{N}_T^*)}^{2} \\
+
\sum_{S \in \Sides_T} \|\bar{u}-\bar{u}_\mathscr{T}\|_{H^{-2}(\mathcal{N}_S)}^{2}
+
{\rm{osc}}^2_\T(y_d; \mathcal{N}_T^*),
\end{multline*}
where $\mathcal{N}^*_T$ is defined in \eqref{NT_star} and the hidden constant is independent of the optimal variables, their approximations, the size of the elements in the mesh $\mathscr{T}$, and its cardinality $\#\mathscr{T}$.
\end{theorem}


\section{Numerical Examples}\label{ne}
In this section we conduct a series of numerical examples that illustrate the performance of the error estimator that we designed and analyzed in section \ref{global_estimator}. In Section \ref{sub:example_3} below, we go beyond the presented analysis and perform a numerical experiment where the assumption of the convexity of the domain is violated. The presented numerical examples have been carried out with the help of a code that we implemented using MATLAB (R2017a). All matrices have been assembled exactly. The right hand sides as well as the  approximation errors are computed by a quadrature formula which is exact for polynomials of degree 19.

\footnotesize{
\begin{algorithm}[ht]
\caption{\textbf{Adaptive semismooth Newton algorithm.}}
\label{Algorithm}
\SetKwInput{set}{Set}
\SetKwInput{ase}{Active set strategy}
\SetKwInput{al}{A posteriori error estimation}
\SetKwInput{Input}{Input}
\SetAlgoLined
\Input{Initial mesh $\mathscr{T}_0$, desired state $y_d$ and sparsity parameter $\alpha$.}
\set{$i=0$.}
\ase{}
Compute $[\bar{y}^{}_\mathscr{T},\bar{p}^{}_\mathscr{T},\bar{u}^{}_\mathscr{T}]=\textbf{Newton method}[\mathscr{T},\alpha, y_d]$, which implements the semismooth Newton method of \citep[Section 6]{CCK2012}. 
\\
\al 
\\
For each $T\in\mathscr{T}$ compute the local error indicator $\E^2_{\mathrm{ocp},T}$ given in \eqref{def:total_indicator}.
\\
Mark an element $T$ for refinement if $\E^2_{\mathrm{ocp},T}> \displaystyle\frac{1}{2}\max_{T'\in \mathscr{T}}\E^2_{\mathrm{ocp},T'}$.
\\
From step $\boldsymbol{3}$, construct a new mesh, using a longest edge bisection algorithm. Set $i \leftarrow i + 1$, and go to step $\boldsymbol{1}$.
\end{algorithm}}
\normalsize

\begin{figure}[ht]
\centering
\begin{minipage}{0.22\textwidth}\centering
\includegraphics[width=2cm,height=2cm,scale=0.5]{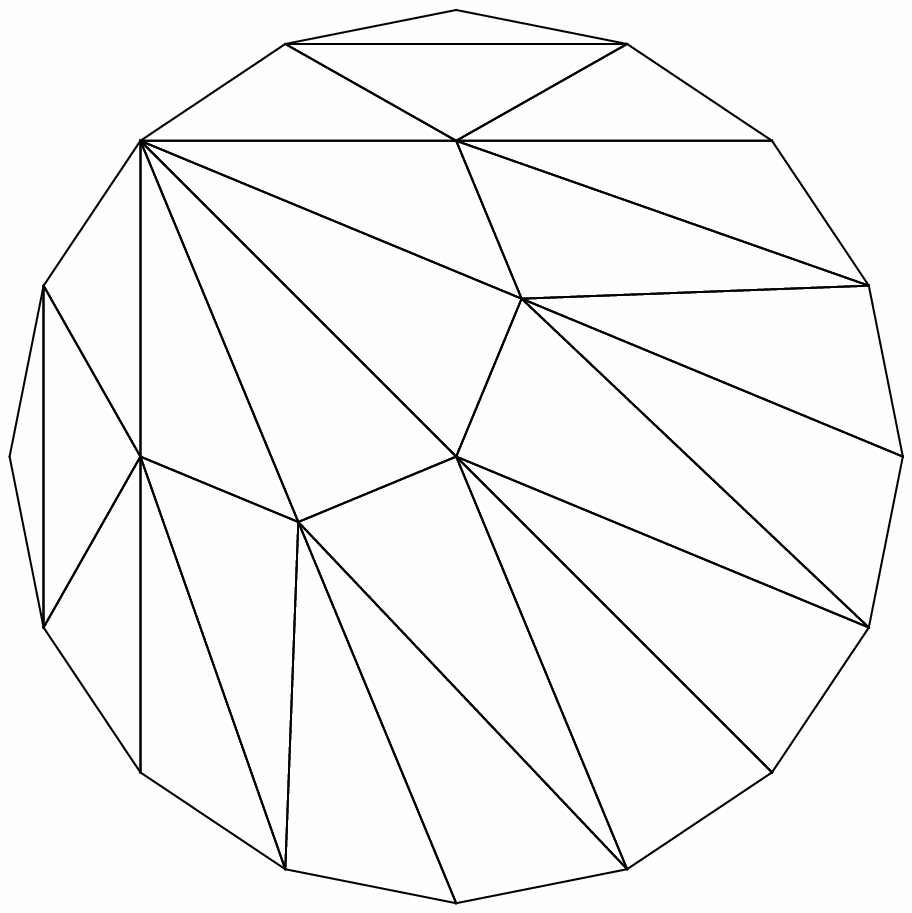}\\
\end{minipage}
\begin{minipage}{0.22\textwidth}\centering
\includegraphics[width=2cm,height=2cm,scale=0.5]{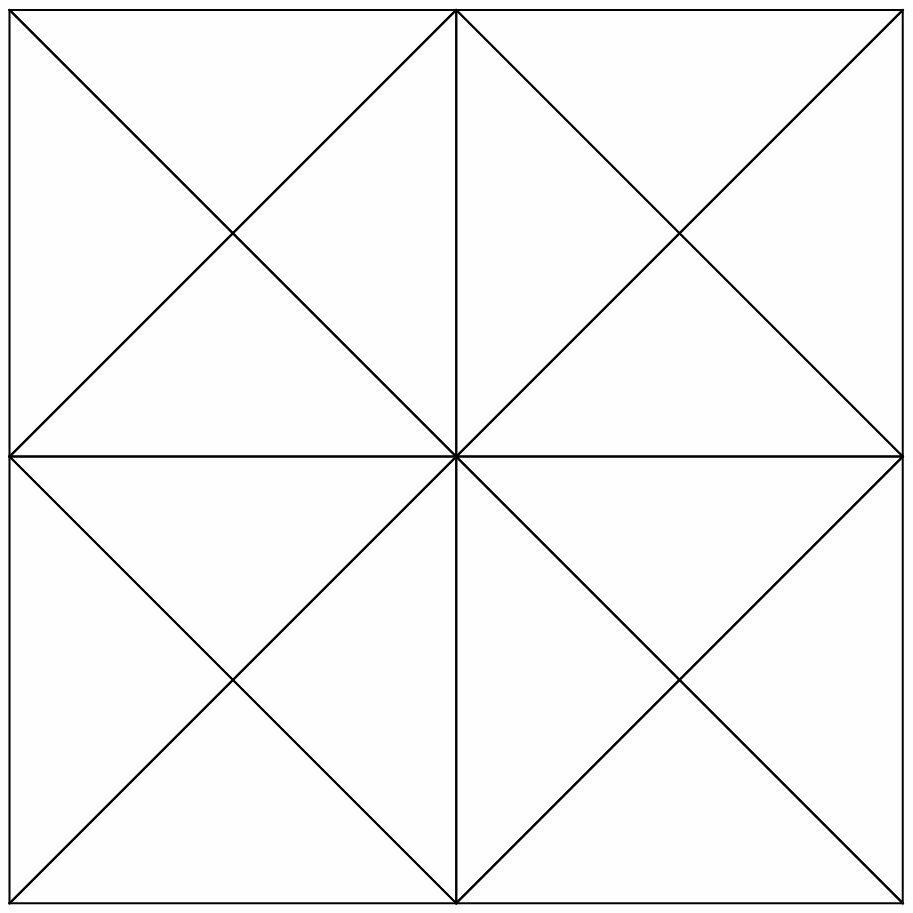}\\
\end{minipage}
\begin{minipage}{0.22\textwidth}\centering
\includegraphics[width=2cm,height=2cm,scale=0.5]{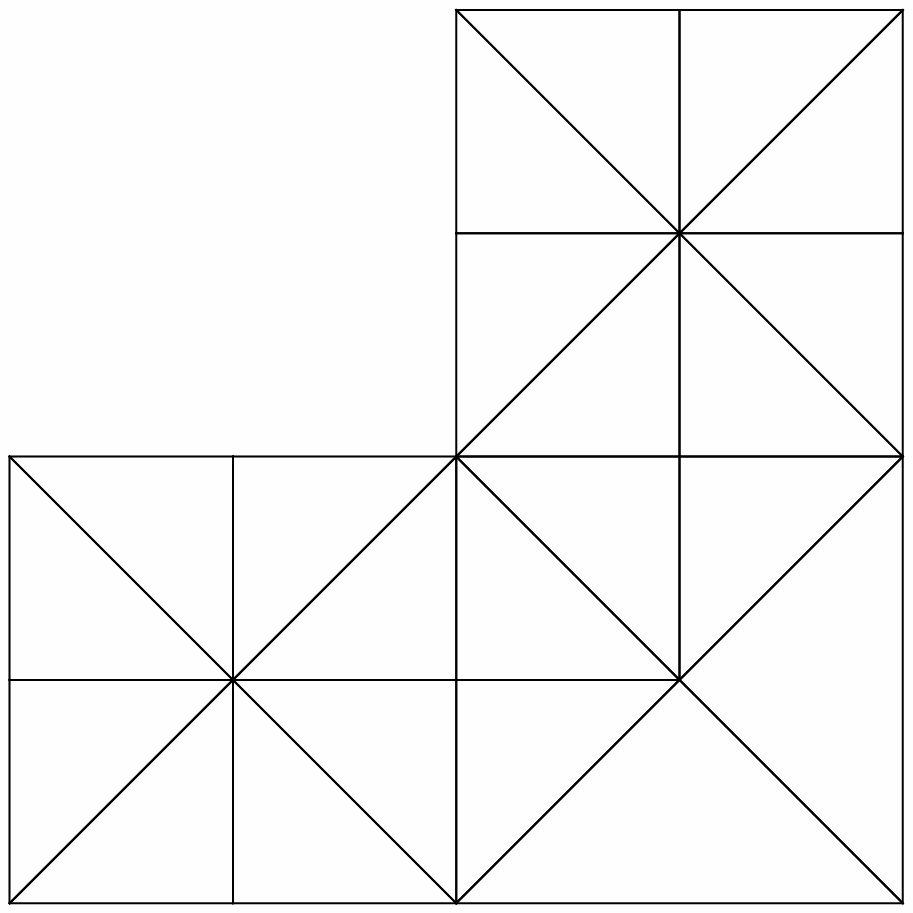}\\
\end{minipage}
\caption{The initial meshes used when the domain $\Omega$ is a circle (Example 1), a square (Example 2) and a L--shaped domain (Example 3).}
\label{fig:initial_meshes}
\end{figure}

For a given partition $\T$, we seek $(\bar{y}_\T,\bar{p}_\T,\bar{u}_\T)\in \mathbb{V}(\T)\times \mathbb{V}(\T)\times \mathbb{U}(\T)$ that solves the discrete optimality system \eqref{eq:discrete_optimality_system}. We solve such a nonlinear system of equations using the semismooth Newton method that was devised in \citep[Section 6]{CCK2012}. Once a discrete solution is obtained, on the basis of \eqref{eq:estimador_bound}, we compute the error indicator
\begin{equation}\label{def:total_indicator}
\E^2_{\mathrm{ocp},T}:= \E_{\rm{ocp}}^2(\bar{y}_\T,\bar{p}_\T,\bar{u}_\T;T)= \E_y^2(\bar{y}_\T,\bar{u}_\T;T)+\E_p^2(\bar{p}_\T,\bar{y}_\T;T),
\end{equation}
which is defined in terms of the local indicators given by \eqref{def:state_indicator} and \eqref{def:adjoint_indicator}, to drive the adaptive procedure described in Algorithm \ref{Algorithm}. For the numerical results, we define the total number of degrees of freedom $\rm{Ndof} = 2\:\rm{dim}(\mathbb{V}(\T)) + \rm{dim}(\mathbb{U}(\T))$, and the errors $e_y := \bar{y}-\bar{y}_\T$ and  $e_p := \bar{p}-\bar{p}_\T$. To assess the accuracy of the approximation, and since $\|\bar{u}-\bar{u}_\T\|_{H^{-2}(\Omega)}$ is not computable, we measure the error in the norm
\[
\|(e_y,e_p)\|_\Omega:= \left (\|e_y\|_{L^2(\Omega)}^2+\|e_p\|_{L^\infty(\Omega)}^2 \right)^{\tfrac{1}{2}}.
\]
The initial meshes for our numerical examples are shown in Figure \ref{fig:initial_meshes}.

We now provide three numerical experiments. First we consider a problem where an exact solution can be obtained: it is based on the solution constructed in \citep[Section 8.1]{PV2013}. In the second example the exact solution is not known. Finally, in the third example, we go beyond the presented analysis and perform a numerical experiment where we violate the assumption of the convexity of the domain.


\subsection{A problem with known exact solution}
\label{sub:example1}

\begin{figure}[ht]
\begin{minipage}{0.32\textwidth}\centering
\psfrag{estimador-t}{$\E_{\rm{ocp}}$}
\psfrag{estimador y}{$\E_y$}
\psfrag{estimador p}{$\E_p$}
\psfrag{Ndofs}{$\textrm{Ndof}$}
\psfrag{Ndofs-pot}{$\footnotesize{\textrm{Ndof}^{-1}}$}
\psfrag{estimators-ex1}{\hspace{-1.2cm}\small{Estimator and its contributions}}
\includegraphics[trim={0 0 0 0},clip,width=4cm,height=4cm,scale=0.6]{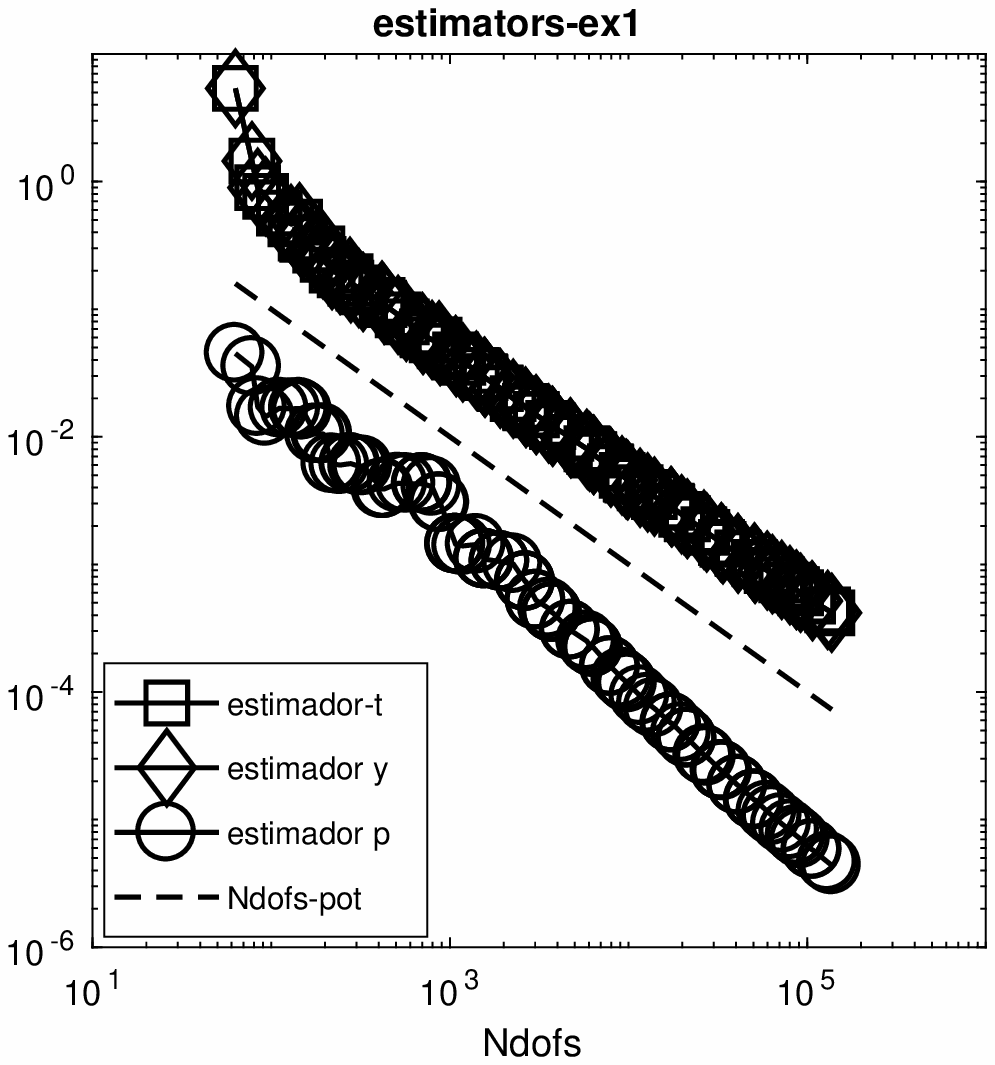}\\\tiny{(a)}
\end{minipage}
\begin{minipage}{0.32\textwidth}\centering
\psfrag{error y}{$\|e_y\|_{L^2(\Omega)}$}
\psfrag{error p}{$\|e_p\|_{L^\infty(\Omega)}$}
\psfrag{Ndofs}{$\textrm{Ndof}$}
\psfrag{Ndofs-pot}{$\footnotesize{\textrm{Ndof}^{-1}}$}
\psfrag{errors-ex1}{\hspace{-1cm}\small{Approximation errors}}
\includegraphics[trim={0 0 0 0},clip,width=4cm,height=4cm,scale=0.6]{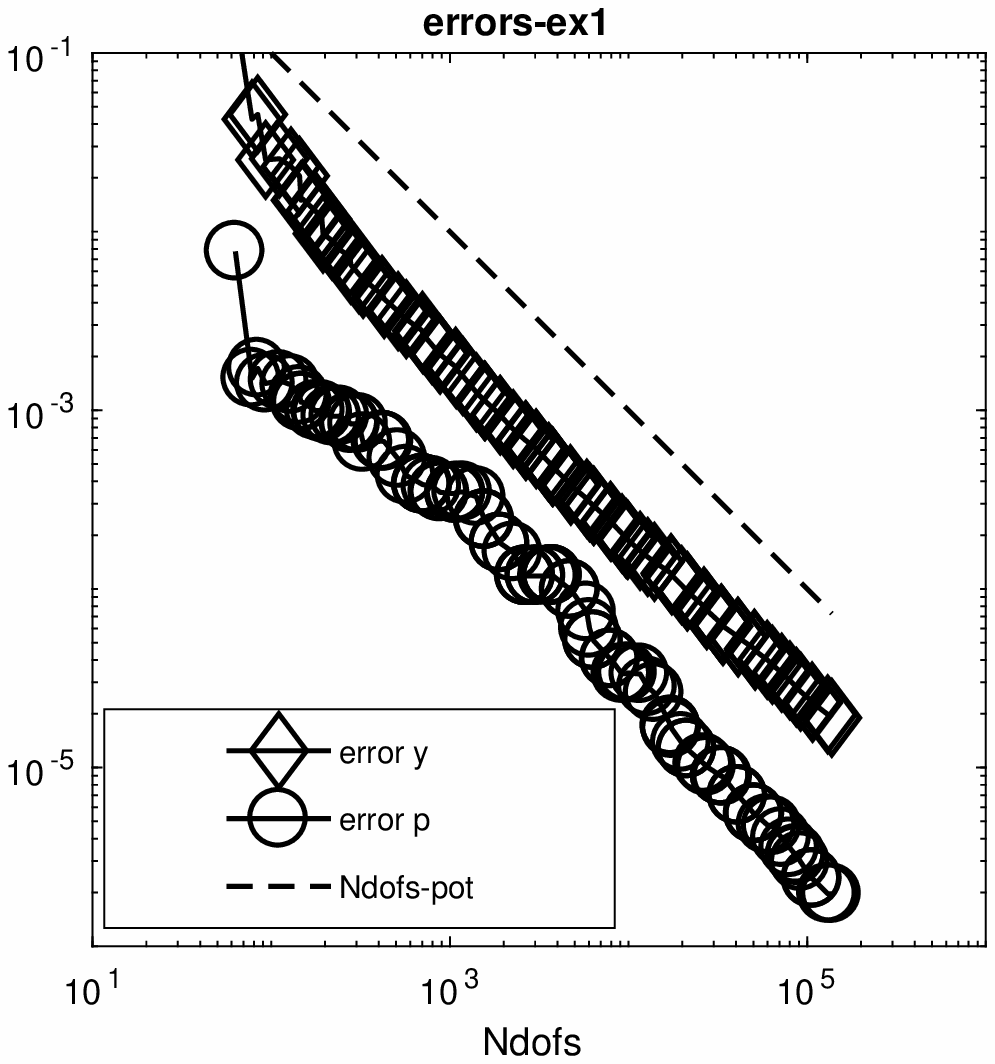}\\\tiny{(b)}
\end{minipage}
\begin{minipage}{0.32\textwidth}\centering
\psfrag{error y unif}{$\|e_y\|_{L^2(\Omega)}$}
\psfrag{error y adap}{$\|e_y\|_{L^2(\Omega)}$}
\psfrag{Ndofs}{$\textrm{Ndof}$}
\psfrag{Ndofs-pot}{$\footnotesize{\textrm{Ndof}^{-\frac{1}{2}}}$}
\psfrag{Ndofs-pot-1}{$\footnotesize{\textrm{Ndof}^{-1}}$}
\psfrag{errors-ex1}{\hspace{-1.6cm}\small{Adaptive vs uniform refinement}}
\includegraphics[trim={0 0 0 0},clip,width=4cm,height=4cm,scale=0.6]{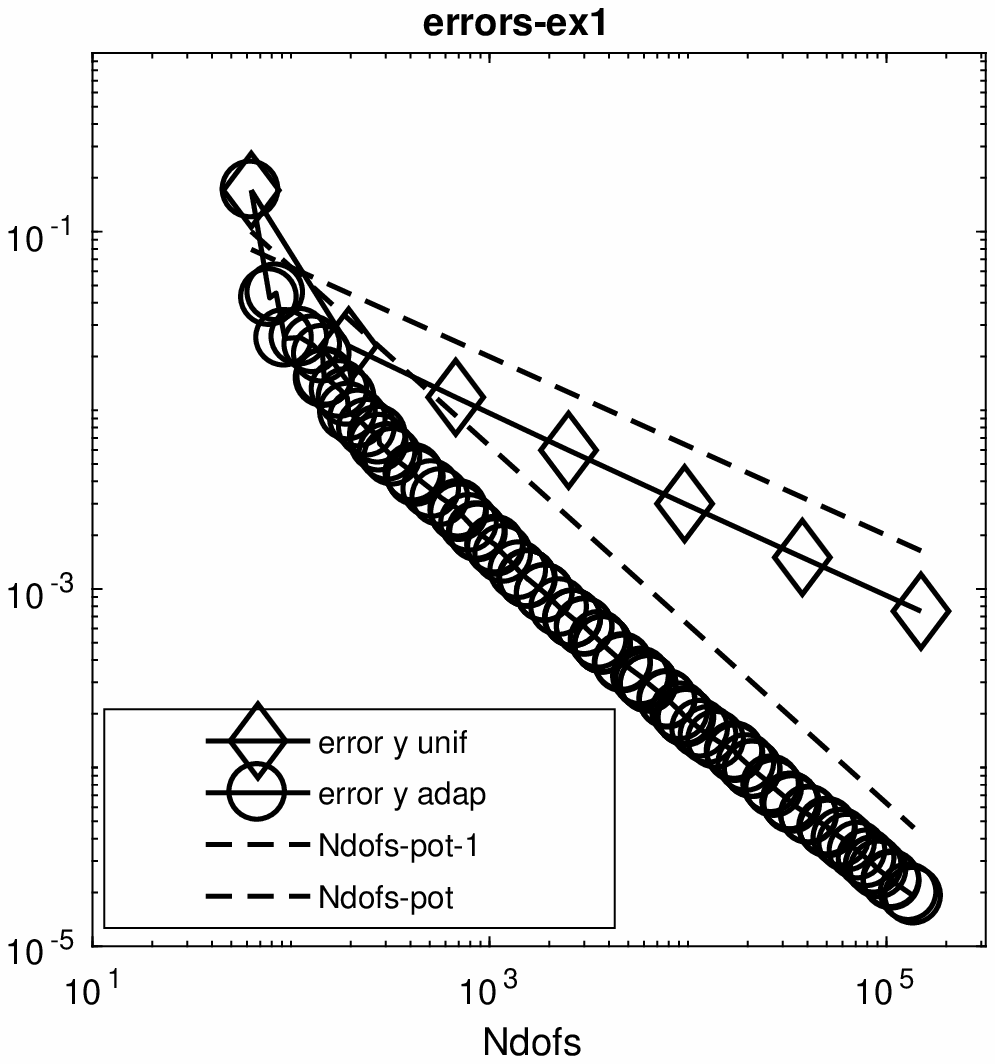}\\\tiny{(c)}
\end{minipage}
\\
\begin{minipage}[c]{0.3\textwidth}\centering
\psfrag{Uh}{$\bar{u}_\T$}
\includegraphics[trim={0 0 0 0},clip,width=3.7cm,height=3.7cm,scale=0.6]{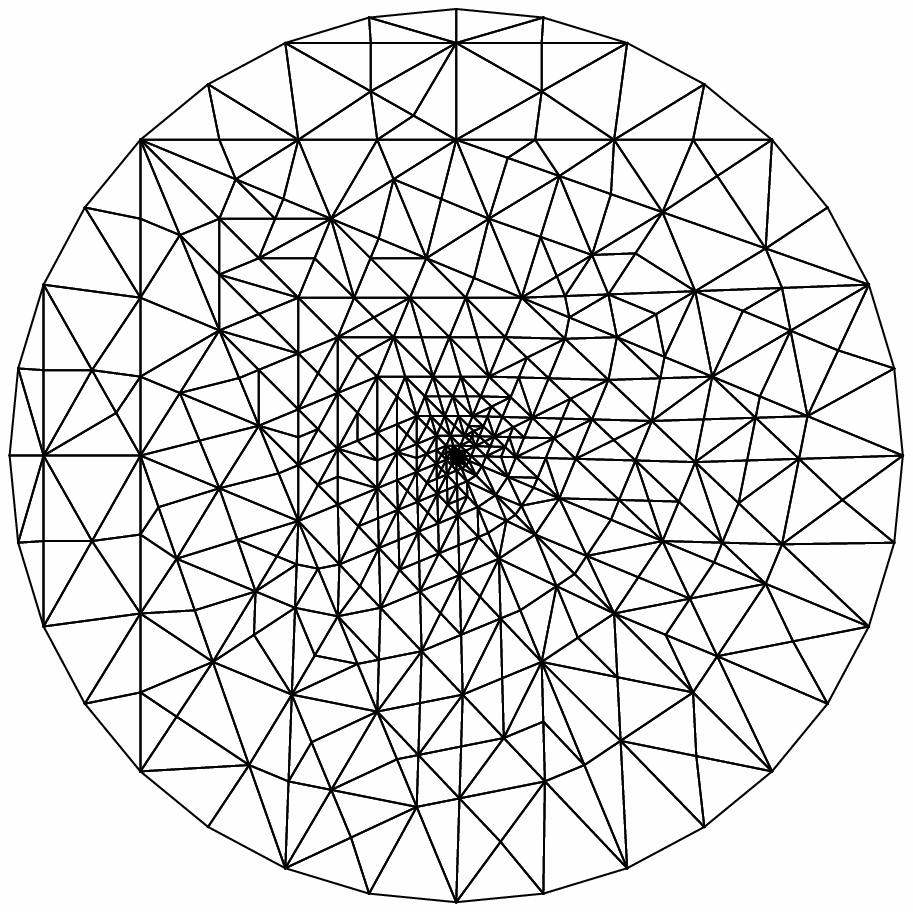}\\\tiny{(d)}
\end{minipage}
\begin{minipage}[c]{0.3\textwidth}\centering
\psfrag{Uh}{$\bar{u}_\T$}
\includegraphics[trim={0 0 0 0},clip,width=4cm,height=4cm,scale=0.5]{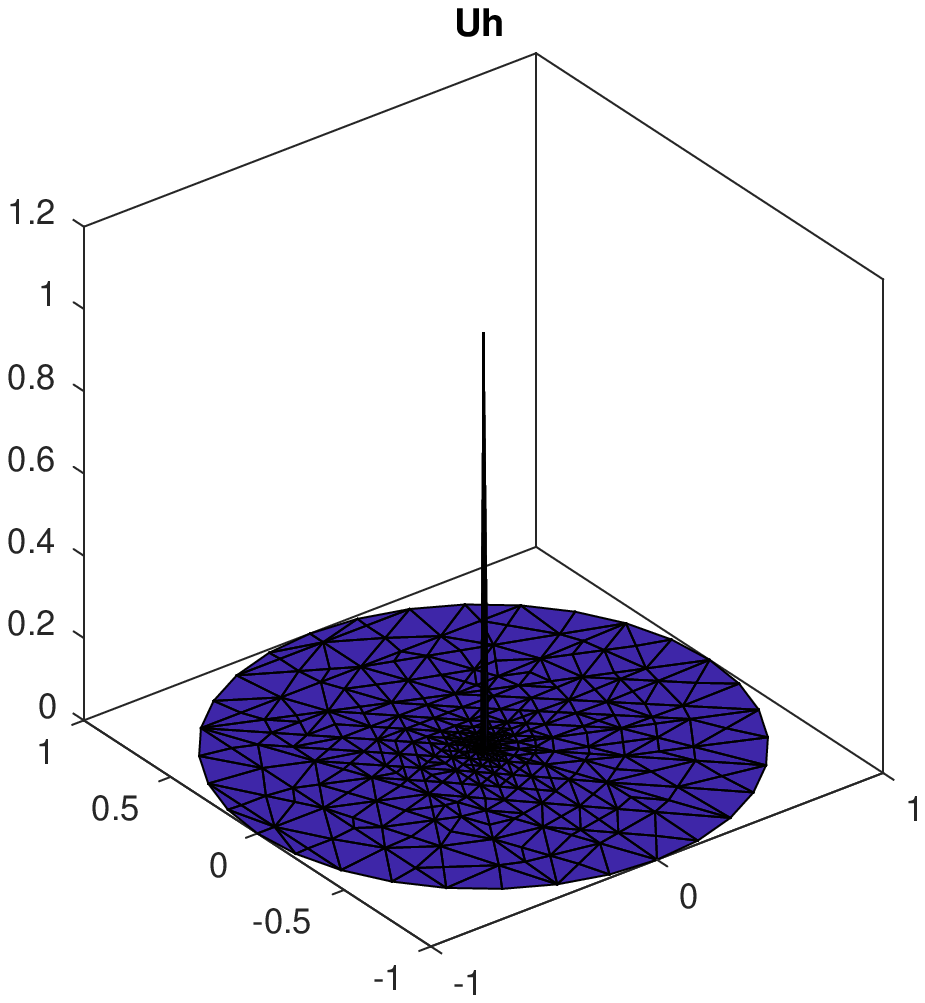}\\\tiny{(e)}
\end{minipage}
\caption{An example with known exact solution; see section~\ref{sub:example1}. (a) Experimental rates of convergence for $\E_{\rm{ocp}}(\bar{y}_\T,\bar{p}_\T,\bar{u}_\T;\T)$, and each of his contributions. (b) Experimental rates of convergence for the errors $e_y$ and $e_p$. (c) Experimental rates of convergence for $e_y$ on uniform and adaptive refinement.
(d) The $30th$ adaptively refined mesh. (e) The approximate discrete solution $\bar{u}_\T$.}
\label{ex_1}
\end{figure}

For this problem we set $\Omega=B_0(1) \subset \mathbb{R}^2$, that is, the unit two dimensional ball, and $\alpha=10^{-2}$. Following \citep[Section 8.1]{PV2013}, the exact optimal state, adjoint state, and exact optimal control are given by
\begin{equation*}
\bar{y}=-\frac{1}{2\pi}\ln{(|x|)},\qquad \bar{p}=-0.02|x|^3+0.03|x|^2-0.01, \qquad \bar u = \delta_0,
\end{equation*}
where for a vector $\xi \in \mathbb{R}^2$ we denote by $|\xi|$ its Euclidean norm. Note that the optimal state $\bar{y}$ is simply a Green's function. On the basis of the solution, we compute the desired state via
\[
  y_d = \Delta \bar p + \bar y = -\alpha \frac{6(3|x|^2-2|x|)}{|x|} - \frac1{2\pi}\ln(|x|);
\]
see \citep[Section 8.1]{PV2013} for details. The results for this experiment are presented in Figure~\ref{ex_1}, where we show the experimental rates of convergence for the error estimator $\E_{\rm{ocp}}$, as well as for each one of its contributions. Since, in this case, we have access to the exact solution, we also compute the decay rates of the errors $\|e_y\|_{L^2(\Omega)}$ and $\|e_p\|_{L^\infty(\Omega)}$. It can be observed that the AFEM described Algorithm \ref{Algorithm} outperforms the FEM of \cite{CCK2012} since it delivers optimal experimental rates of convergence when the latter cannot.

\subsection{An example with no analytical solution}
\label{sub:example_2}

\begin{figure}[ht]
\begin{minipage}{0.32\textwidth}\centering
\psfrag{alpha-0.1-t}{$\E_{\rm{ocp}}$}
\psfrag{alpha-0.1-y}{$\E_y$}
\psfrag{alpha-0.1-p}{$\E_p$}
\psfrag{Ndofs}{$\textrm{Ndof}$}
\psfrag{Ndofs-pot}{$\footnotesize{\textrm{Ndof}^{-1}}$}
\psfrag{estimators-ex1}{\hspace{-1.6cm}\small{Estimator contributions for $\alpha=10^{-1}$}}
\includegraphics[trim={0 0 0 0},clip,width=4cm,height=4cm,scale=0.6]{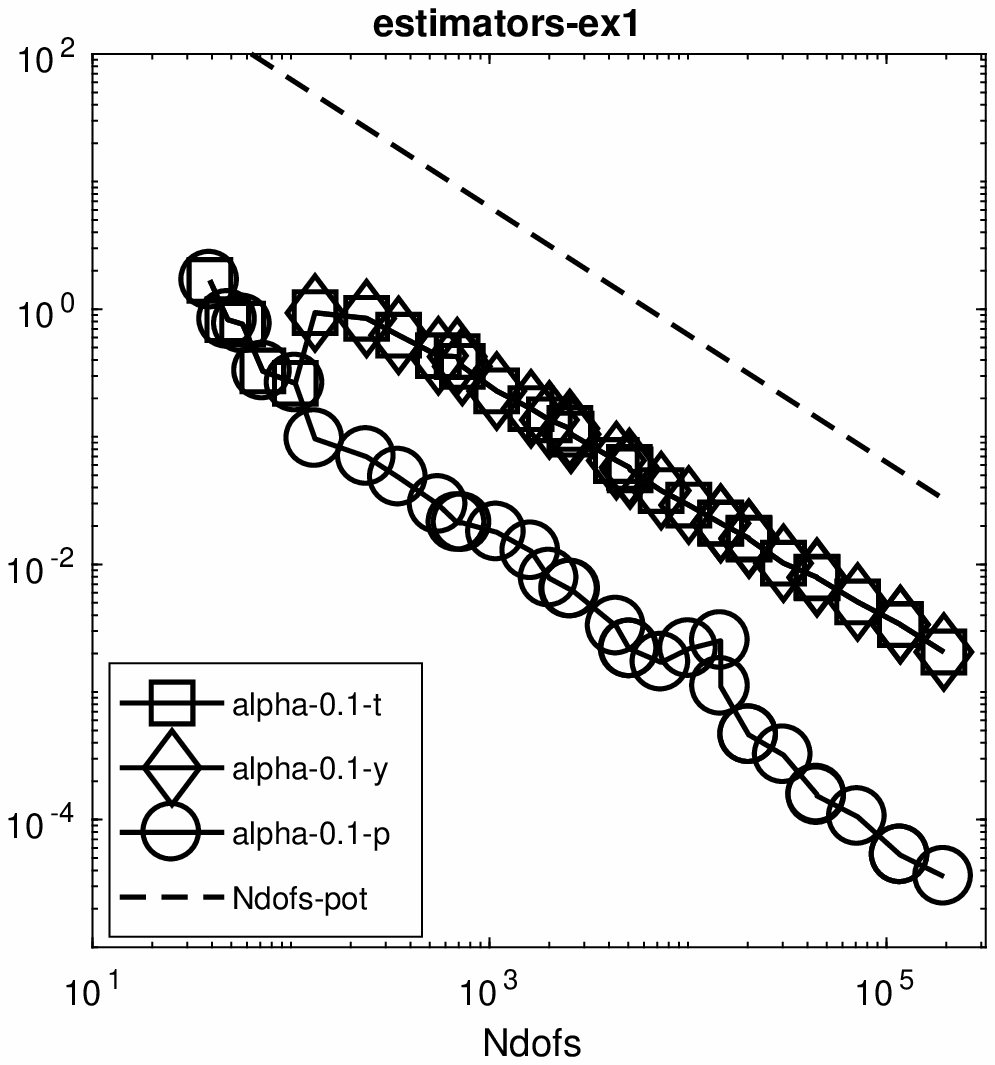}\\\tiny{(a)}
\end{minipage}
\begin{minipage}{0.32\textwidth}\centering
\psfrag{alpha-0.01-t}{$\E_{\rm{ocp}}$}
\psfrag{alpha-0.01-y}{$\E_y$}
\psfrag{alpha-0.01-p}{$\E_p$}
\psfrag{Ndofs}{$\textrm{Ndof}$}
\psfrag{Ndofs-pot}{$\footnotesize{\textrm{Ndof}^{-1}}$}
\psfrag{estimators-ex1}{\hspace{-1.6cm}\small{Estimator contributions for $\alpha=10^{-2}$}}
\includegraphics[trim={0 0 0 0},clip,width=4cm,height=4cm,scale=0.6]{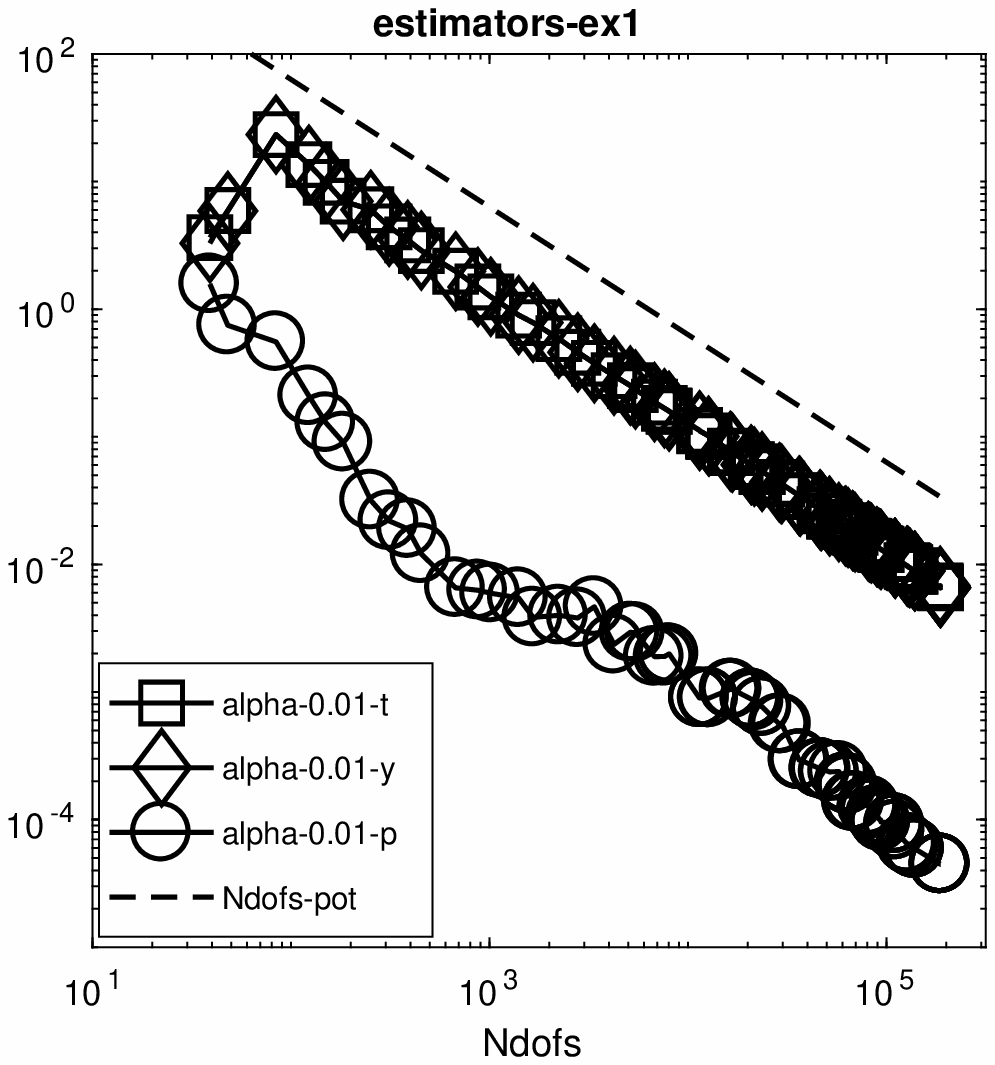}\\\tiny{(b)}
\end{minipage}
\begin{minipage}{0.32\textwidth}\centering
\psfrag{alpha-0.001-t}{$\E_{\rm{ocp}}$}
\psfrag{alpha-0.001-y}{$\E_y$}
\psfrag{alpha-0.001-p}{$\E_p$}
\psfrag{Ndofs}{$\textrm{Ndof}$}
\psfrag{Ndofs-pot}{$\footnotesize{\textrm{Ndof}^{-1}}$}
\psfrag{estimators-ex1}{\hspace{-1.6cm}\small{Estimator contributions for $\alpha=10^{-3}$}}
\includegraphics[trim={0 0 0 0},clip,width=4cm,height=4cm,scale=0.6]{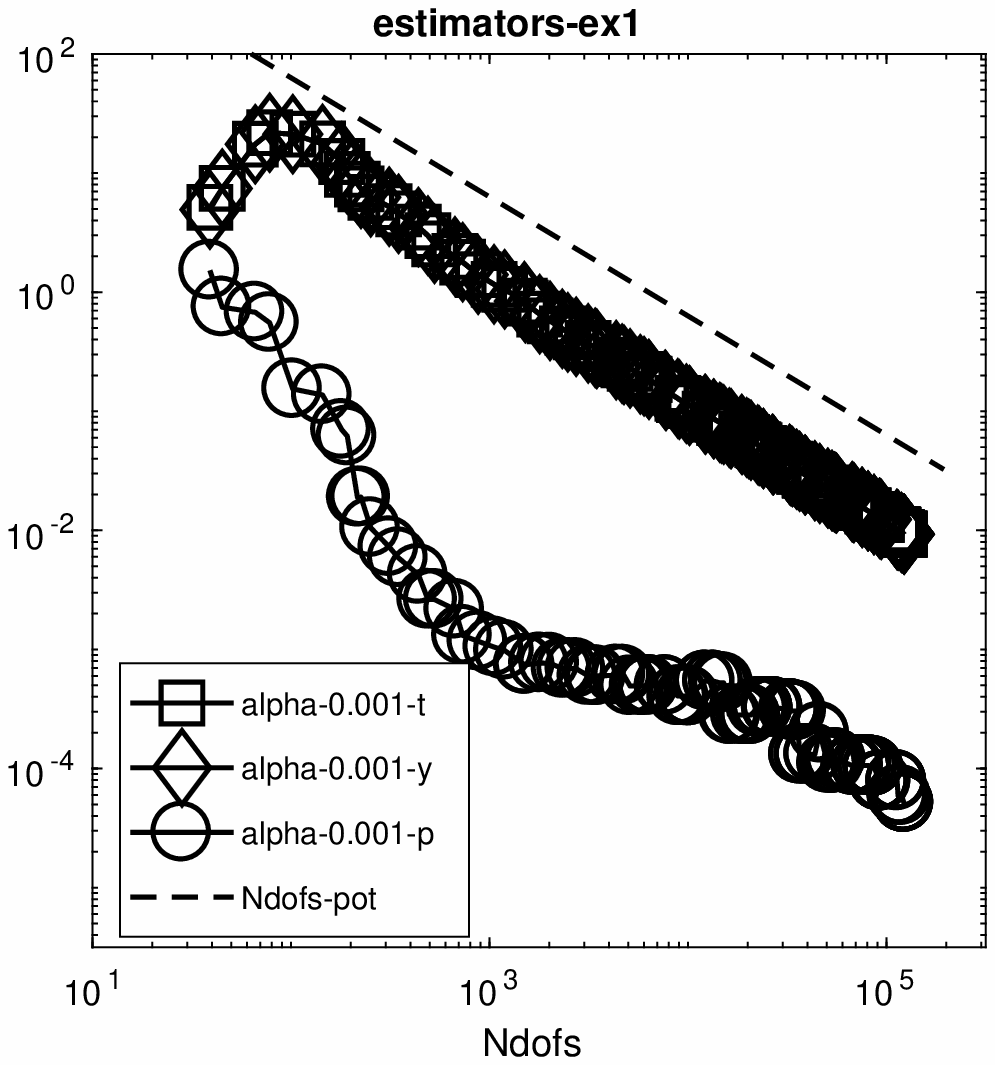}\\\tiny{(c)}
\end{minipage}
\\
\begin{minipage}{0.32\textwidth}\centering
\psfrag{alpha-0.0001-t}{$\E_{\rm{ocp}}$}
\psfrag{alpha-0.0001-y}{$\E_y$}
\psfrag{alpha-0.0001-p}{$\E_p$}
\psfrag{Ndofs}{$\textrm{Ndof}$}
\psfrag{Ndofs-pot}{$\footnotesize{\textrm{Ndof}^{-1}}$}
\psfrag{estimators-ex1}{\hspace{-1.6cm}\small{Estimator contributions for $\alpha=10^{-4}$}}
\includegraphics[trim={0 0 0 0},clip,width=4cm,height=4cm,scale=0.6]{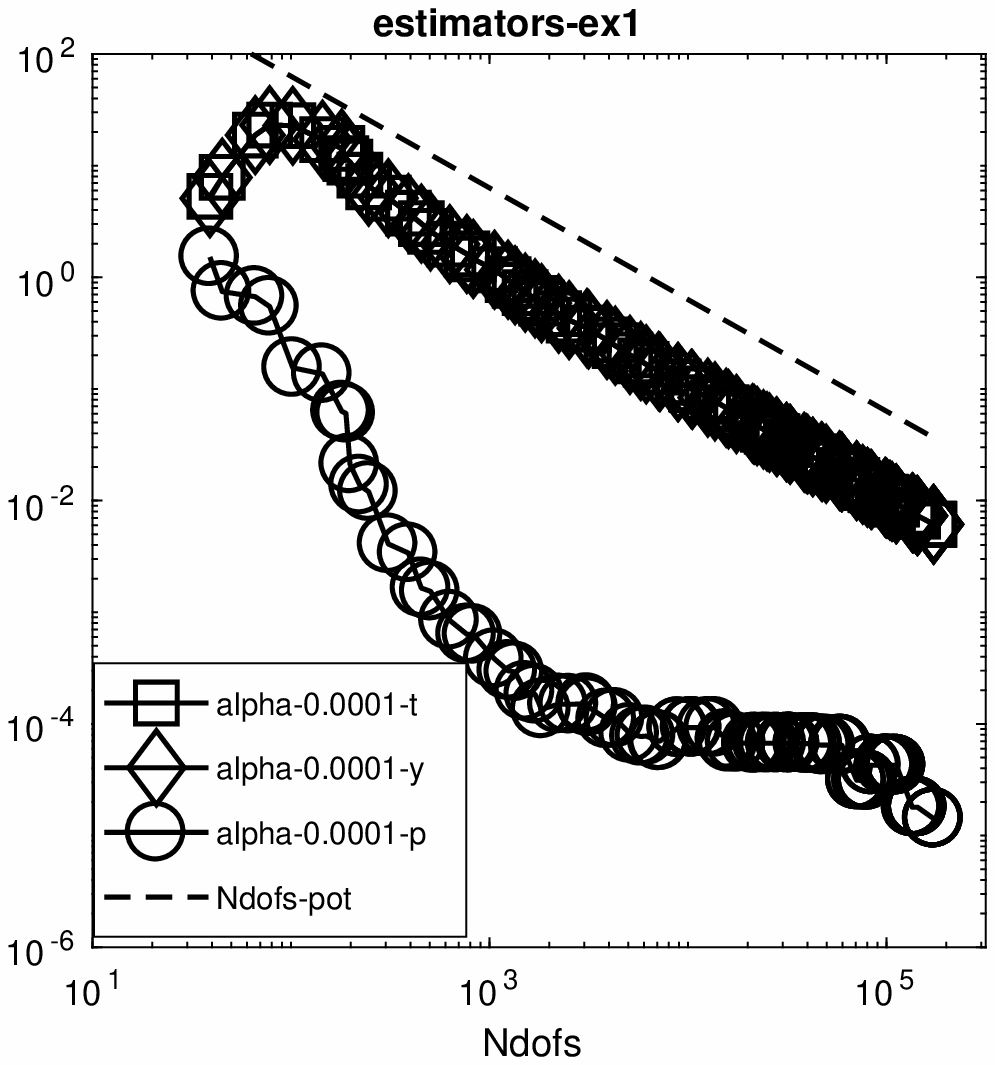}\\\tiny{(d)}
\end{minipage}
\begin{minipage}{0.32\textwidth}\centering
\psfrag{alpha-0.00001-t}{$\E_{\rm{ocp}}$}
\psfrag{alpha-0.00001-y}{$\E_y$}
\psfrag{alpha-0.00001-p}{$\E_p$}
\psfrag{Ndofs}{$\textrm{Ndof}$}
\psfrag{Ndofs-pot}{$\footnotesize{\textrm{Ndof}^{-1}}$}
\psfrag{estimators-ex1}{\hspace{-1.6cm}\small{Estimator contributions for $\alpha=10^{-5}$}}
\includegraphics[trim={0 0 0 0},clip,width=4cm,height=4cm,scale=0.6]{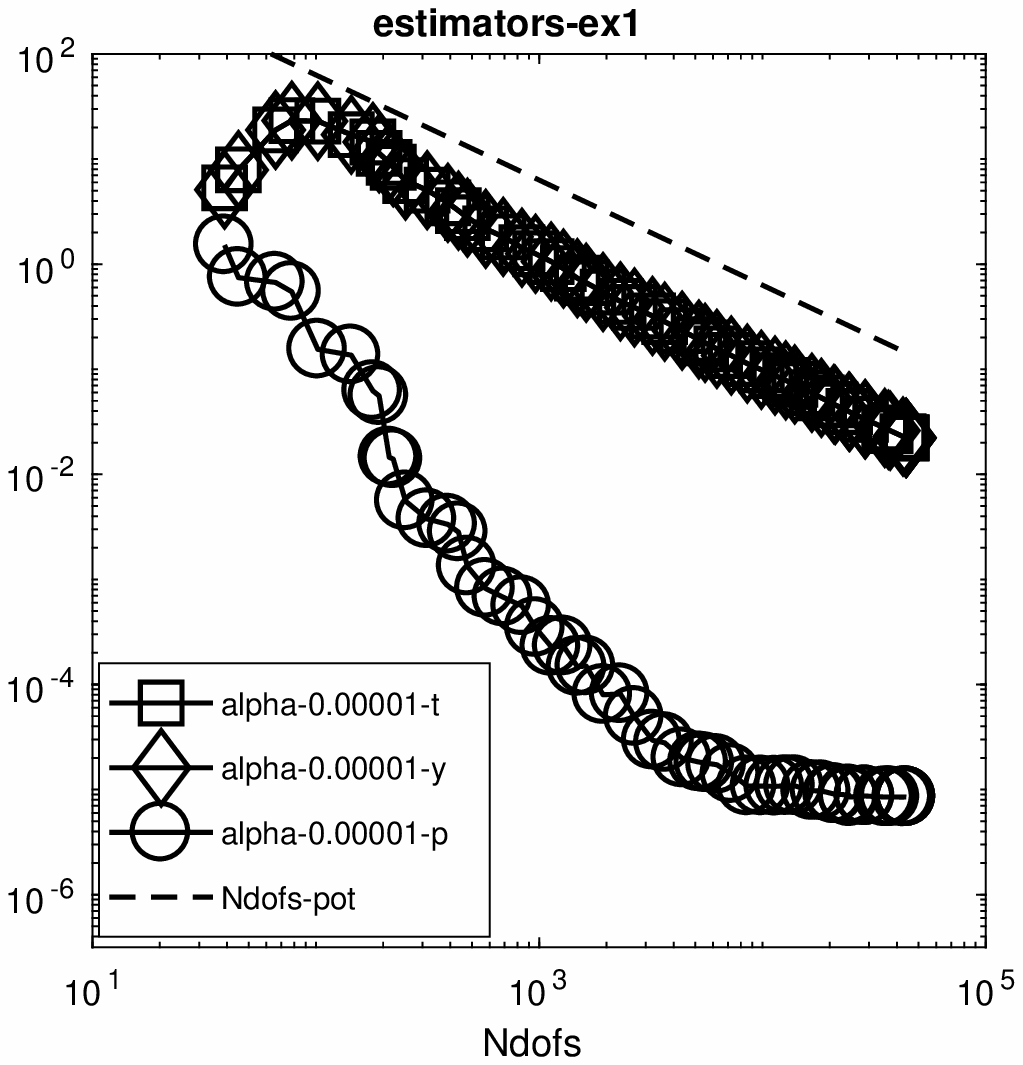}\\\tiny{(e)}
\end{minipage}
\begin{minipage}{0.32\textwidth}\centering
\psfrag{alpha-0.000001-t}{$\E_{\rm{ocp}}$}
\psfrag{alpha-0.000001-y}{$\E_y$}
\psfrag{alpha-0.000001-p}{$\E_p$}
\psfrag{Ndofs}{$\textrm{Ndof}$}
\psfrag{Ndofs-pot}{$\footnotesize{\textrm{Ndof}^{-1}}$}
\psfrag{estimators-ex1}{\hspace{-1.6cm}\small{Estimator contributions for $\alpha=10^{-6}$}}
\includegraphics[trim={0 0 0 0},clip,width=4cm,height=4cm,scale=0.6]{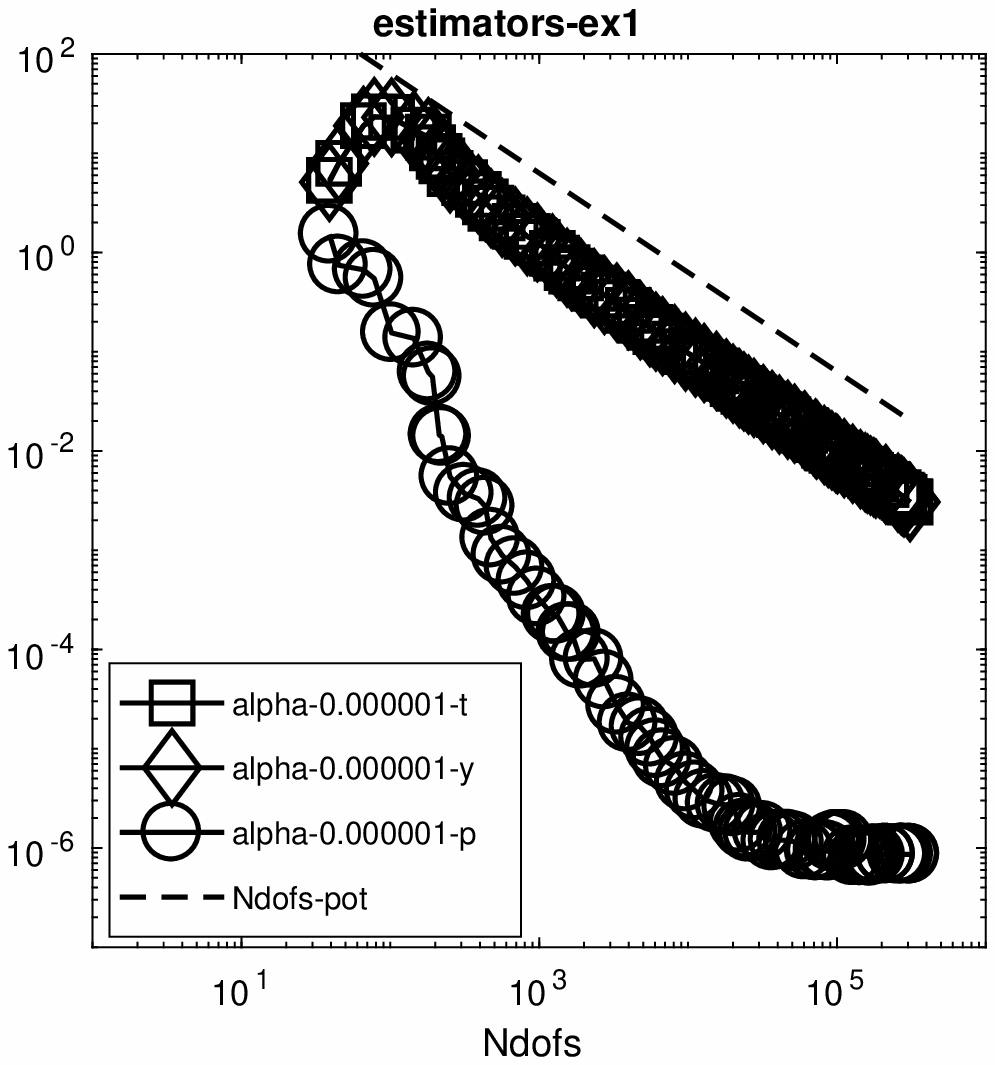}\\\tiny{(f)}
\end{minipage}
\\
\begin{minipage}[c]{0.3\textwidth}\centering
\psfrag{Uh}{$\bar{u}_\T$}
\includegraphics[trim={0 0 0 0},clip,width=3.6cm,height=3.6cm,scale=0.7]{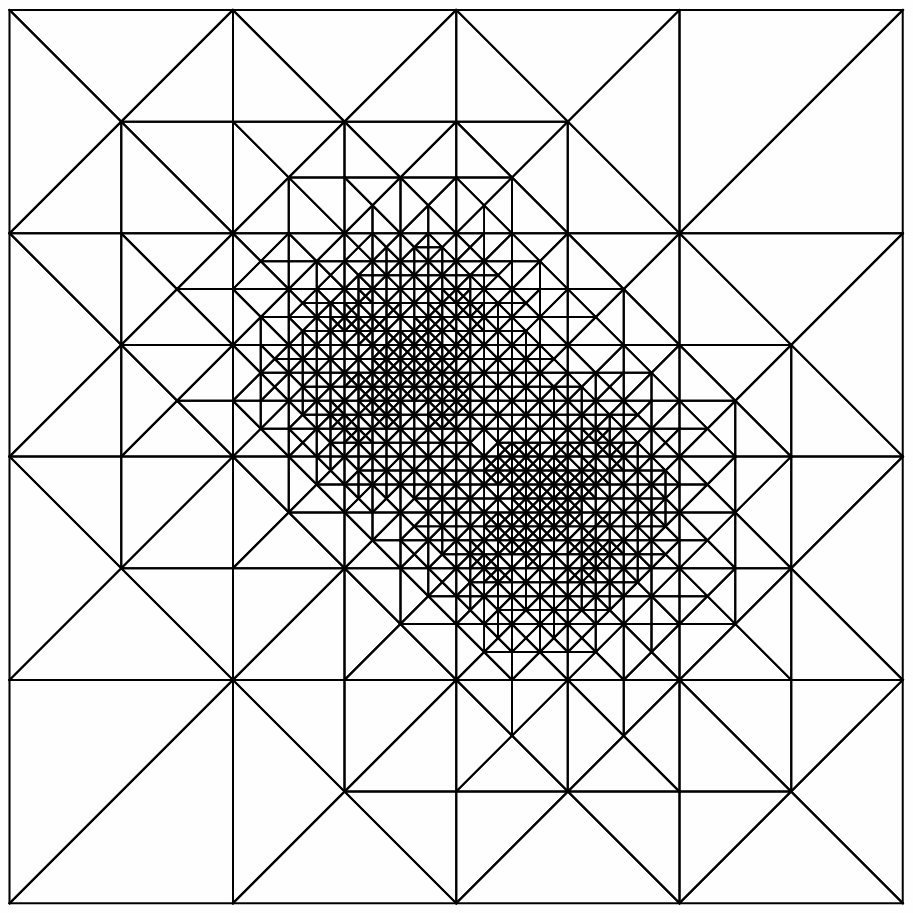}\\\tiny{(g)}
\end{minipage}
\begin{minipage}[c]{0.3\textwidth}\centering
\psfrag{Uh}{$\bar{u}_\T$}
\includegraphics[trim={0 0 0 0},clip,width=4cm,height=4cm,scale=0.5]{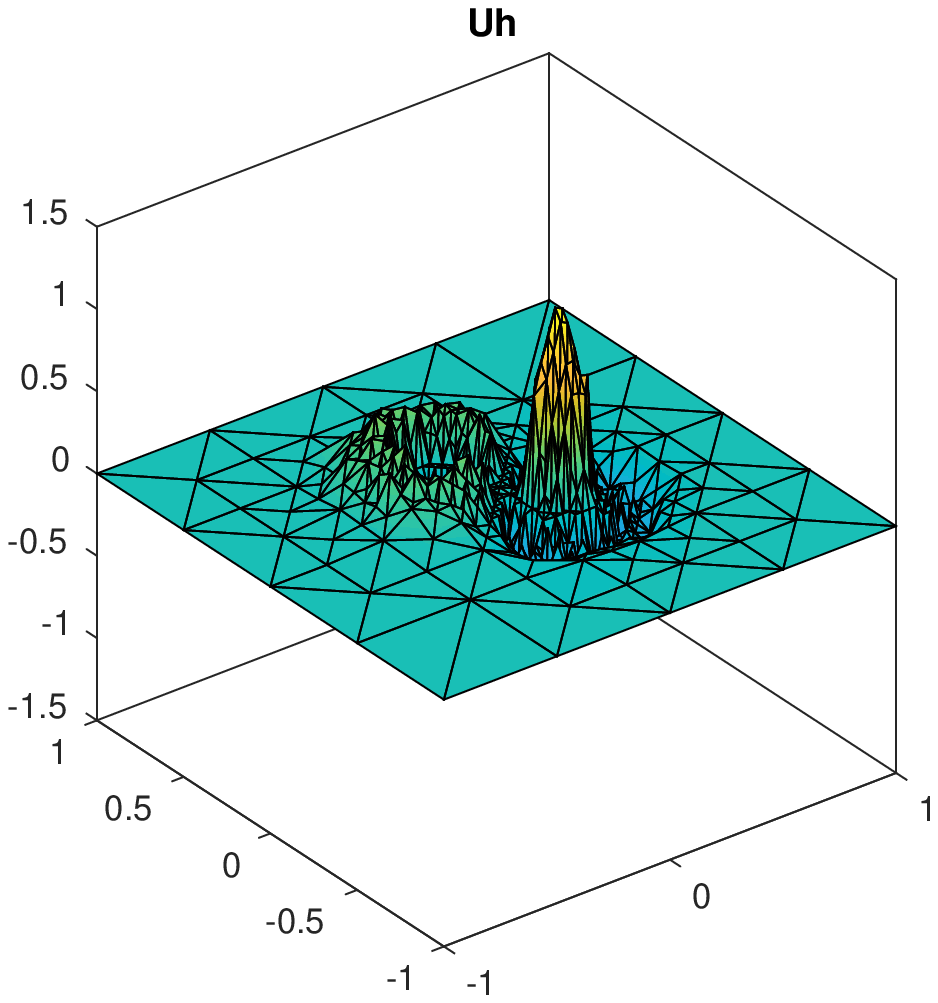}\\\tiny{(h)}
\end{minipage}
\caption{Results for the example of section~\ref{sub:example_2}. (a)--(f) Experimental rates of convergence for $\E_{\rm{ocp}}(\bar{y}_\T,\bar{p}_\T,\bar{u}_\T;\T)$ and its contributions for $\alpha \in \{10^{-1},10^{-2},10^{-3},10^{-4},10^{-5},10^{-6} \}$, respectively. (g) The $25th$ adaptively refined mesh for $\alpha = 10^{-4}$. (h) The discrete control $\bar{u}_\T$ on the $25th$ adaptively refined mesh for $\alpha = 10^{-4}$.}
\label{ex_2}
\end{figure}

In this case we set $\Omega=(-1,1)^2$ and for $(x_1,x_2) \in \Omega$:
\begin{multline}
  y_d=10\big(
            \exp{(-50\{(x_1-0.2)^2+(x_2+0.1)^2)\}} 
            \\
            - 
            \exp{(-50\{(x_1+0.1)^2+(x_2-0.2)^2)\}}
    \big).
\end{multline}
The purpose of this example is to investigate the effect of varying the sparsity parameter $\alpha$. We consider
\[
\alpha \in \{10^{-1}, 10^{-2},10^{-3},10^{-4},10^{-5},10^{-6}\}.
\]
The results of this experiment are presented in Figure~\ref{ex_2}. We observe that optimal experimental rates of convergence are obtained for the error estimator $\E_{\rm{ocp}}$ and its individual contributions. For $\alpha=10^{-4}$, we display the adaptive mesh obtained after 25 iterations of the designed AFEM and the obtained discrete optimal control variable.

\subsection{An example on a nonconvex domain}
\label{sub:example_3}
We let $\Omega=(-1,1)^2\setminus [0,1) \times (-1,0]$, i.e., an $L$ shaped domain and $\alpha=5\cdot 10^{-3}$. The desired state $y_d$ is given by 
\[
y_d=-\log\big(\sqrt{(x_1-0.2)^2+(x_2+0.2)^2}\big), \quad (x_1,x_2) \in \Omega
\]

\begin{figure}[ht]
\begin{minipage}{0.32\textwidth}\centering
\psfrag{estimador-t}{$\E_{\rm{ocp}}$}
\psfrag{estimador y}{$\E_y$}
\psfrag{estimador p}{$\E_p$}
\psfrag{Ndofs}{$\textrm{Ndof}$}
\psfrag{Ndofs-pot}{$\footnotesize{\textrm{Ndof}^{-1}}$}
\psfrag{estimators-ex1}{\hspace{-1.2cm}\small{Estimator and its contributios}}
\includegraphics[trim={0 0 0 0},clip,width=4cm,height=4cm,scale=0.6]{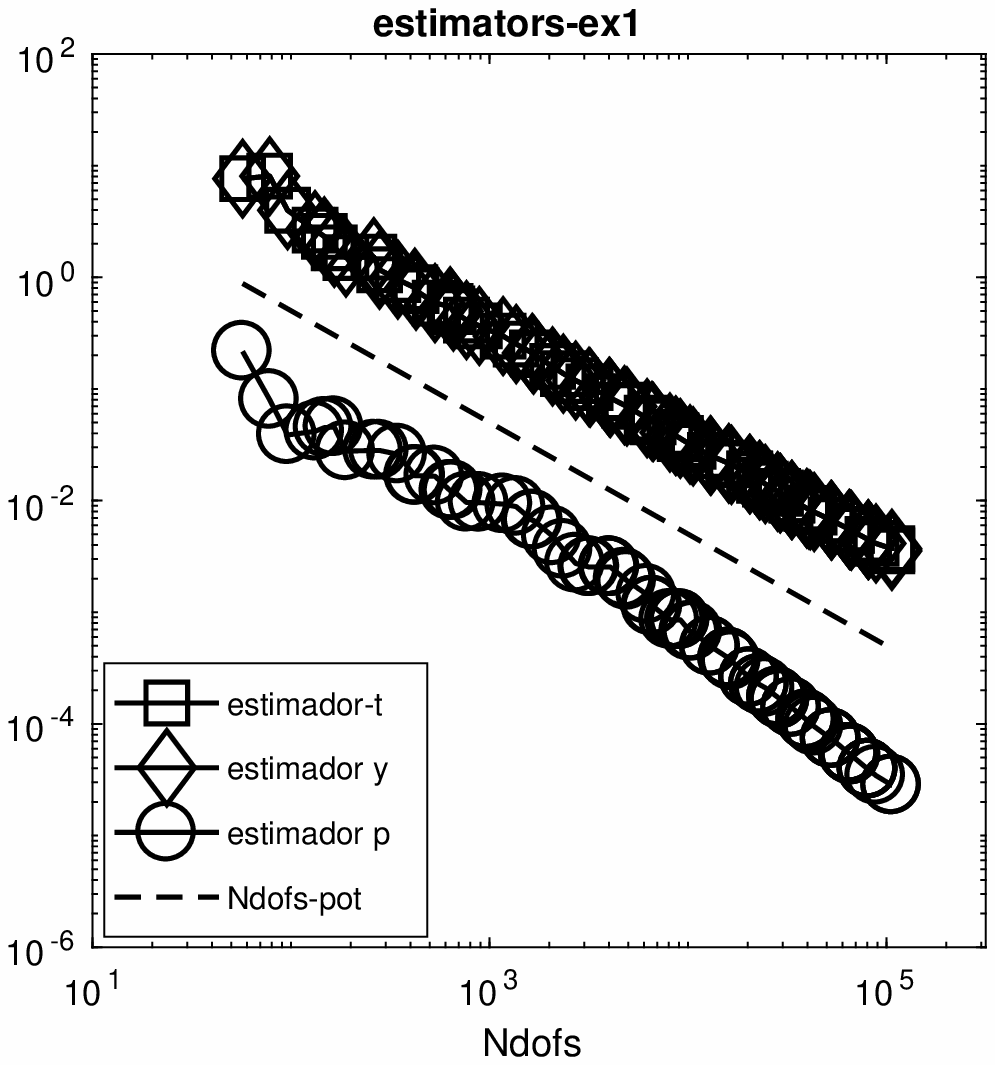}\\\tiny{(a)}
\end{minipage}
\begin{minipage}[c]{0.3\textwidth}\centering
\psfrag{Uh}{$\bar{u}_\T$}
\includegraphics[trim={0 0 0 0},clip,width=3.6cm,height=3.6cm,scale=0.7]{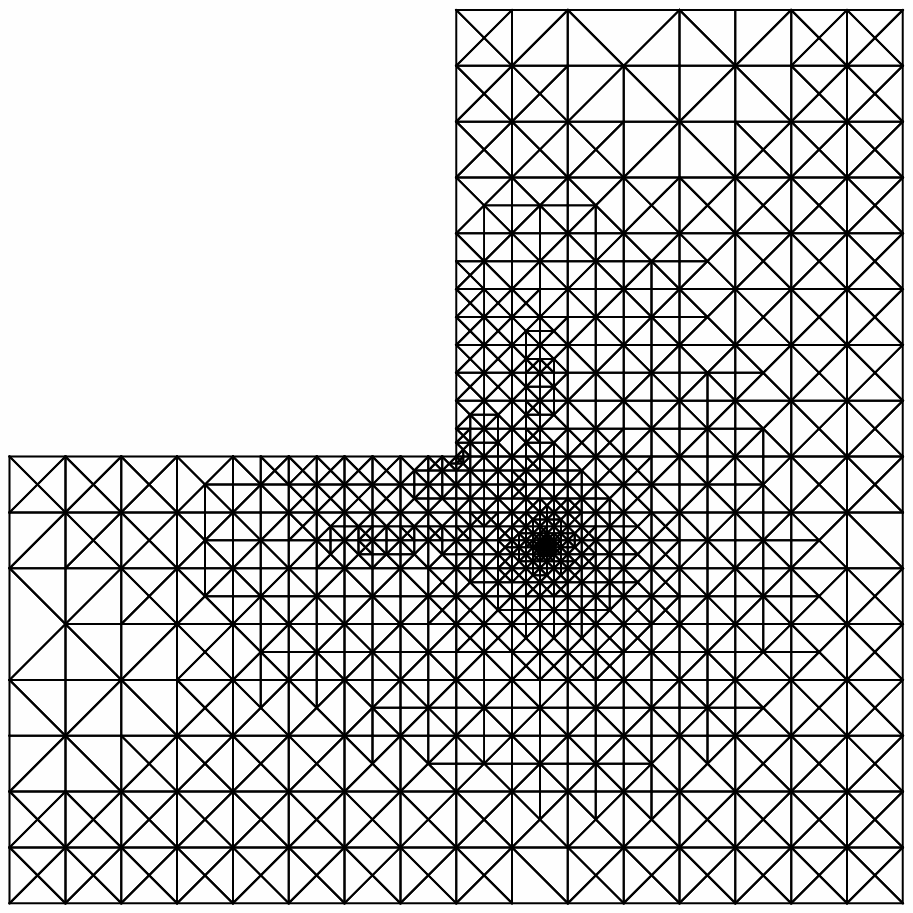}\\\tiny{(b)}
\end{minipage}
\begin{minipage}[c]{0.3\textwidth}\centering
\psfrag{Uh}{$\bar{u}_\T$}
\includegraphics[trim={0 0 0 0},clip,width=4cm,height=4cm,scale=0.5]{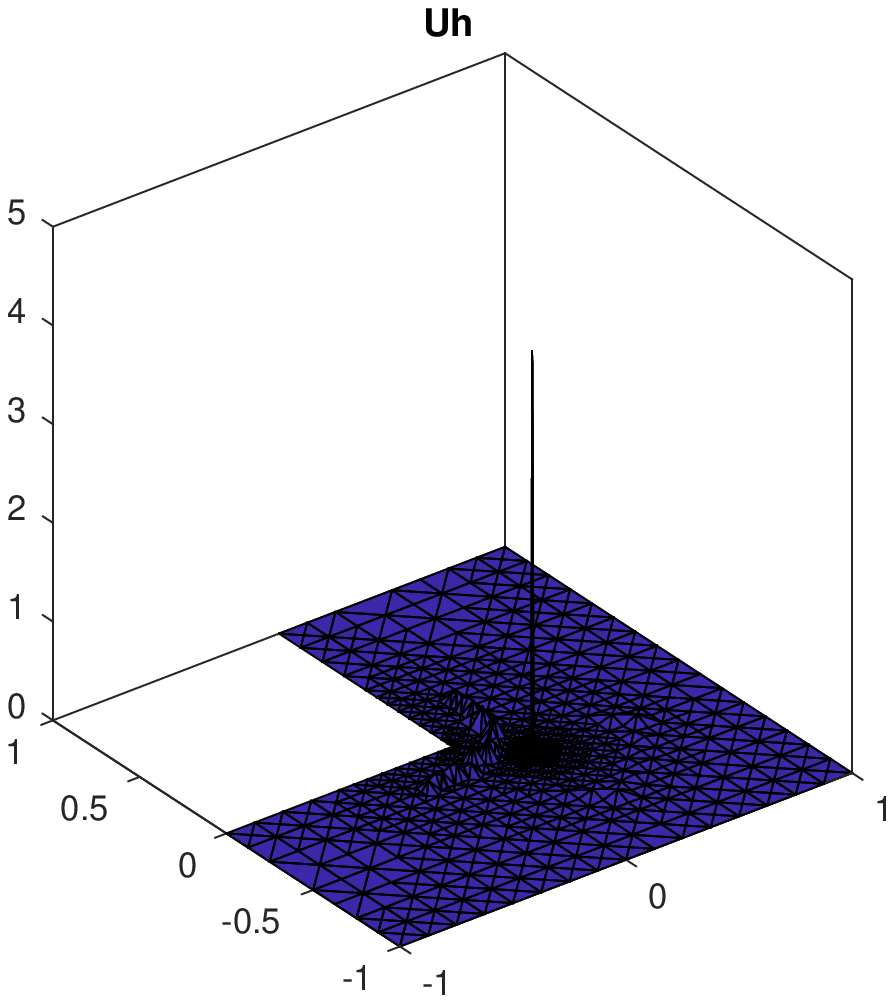}\\\tiny{(c)}
\end{minipage}
\caption{Results for the example of section~\ref{sub:example_3}. (a) Experimental rates of convergence for $\E_{\rm{ocp}}(\bar{y}_\T,\bar{p}_\T,\bar{u}_\T;\T)$, and each of his contributions. (b) The $25th$ adaptively refined mesh. (c) The discrete control $\bar{u}_\T$ on the $25th$ adaptively refined mesh.} 
\label{ex_3}
\end{figure}

The results are shown in Figure \ref{ex_3} where we observe optimal experimental rates of convergence for the proposed a posteriori error estimator $\E_{\rm{ocp}}$ and each one of its contributions. We also observe that the refinement is being concentrated about the singularity exhibited by the optimal control and to a lesser extent the re--entrant corner.
To appreciate the nature of the aforementioned singularity, we also display the computed optimal control variable.

\subsection{Conclusions}
\label{sub:conclusion}

From the presented numerical examples several general conclusions can be drawn: 
\begin{enumerate}[$\bullet$]

\item The error estimator $\E_{\rm{ocp}}$, as well as each of his contributions, exhibit optimal experimental rates of convergence for the experiments that we perform.

\item Most of the refinement occurs near the singularity points, which attests to the efficiency of the devised estimators.

\item The contribution $\E_y(\bar{p}_\T,\bar{y}_\T;\T)$ to the error estimator is the dominating one. We believe that this shows the very singular nature of the problem that defines the state variable. 

\item The third example shows that, even in the presence of a nonconvex domain, the error estimator $\E_{\rm{ocp}}$ exhibits optimal experimental rates of convergence with respect to approximation.
\end{enumerate}

\section*{Acknowledgement}
The authors would like to thank Cesare Molinari (UTFSM) for insightful discussions.

\bibliographystyle{siam}
\footnotesize
\bibliography{biblio}

\begin{thebibliography}{10}

\bibitem{AORS2017}
{\sc A.~Allendes, E.~Ot\'{a}rola, R.~Rankin, and A.~Salgado}, {\em An a
  posteriori error analysis for an optimal control problem with dirac
  measures}, ESAIM Math. Model. Numer. Anal.,  (2018).
\newblock DOI: https://doi.org/10.1051/m2an/m2an170055.

\bibitem{ABR2006}
{\sc R.~Araya, E.~Behrens, and R.~Rodr\'\i~guez}, {\em A posteriori error
  estimates for elliptic problems with {D}irac delta source terms}, Numer.
  Math., 105 (2006), pp.~193--216.

\bibitem{camacho2015}
{\sc F.~Camacho and A.~Demlow}, {\em {$L_2$} and pointwise a posteriori error
  estimates for {FEM} for elliptic {PDE}s on surfaces}, IMA J. Numer. Anal., 35
  (2015), pp.~1199--1227.

\bibitem{casas1986}
{\sc E.~Casas}, {\em Control of an elliptic problem with pointwise state
  constraints}, SIAM J. Control Optim., 24, pp.~1309--1318.

\bibitem{Casas2017}
\leavevmode\vrule height 2pt depth -1.6pt width 23pt, {\em A review on sparse
  solutions in optimal control of partial differential equations}, SeMA J., 74
  (2017), pp.~319--344.

\bibitem{CCK2012}
{\sc E.~Casas, C.~Clason, and K.~Kunisch}, {\em Approximation of elliptic
  control problems in measure spaces with sparse solutions}, SIAM J. Control
  Optim., 50 (2012), pp.~1735--1752.

\bibitem{CHW2012}
{\sc E.~Casas, R.~Herzog, and G.~Wachsmuth}, {\em Approximation of sparse
  controls in semilinear elliptic equations}, in Large-scale scientific
  computing, vol.~7116 of Lecture Notes in Comput. Sci., Springer, Heidelberg,
  2012, pp.~16--27.

\bibitem{CHW:12b}
{\sc E.~Casas, R.~Herzog, and G.~Wachsmuth}, {\em Approximation of sparse
  controls in semilinear equations by piecewise linear functions}, Numer.
  Math., 122 (2012), pp.~645--669.

\bibitem{CHW:12}
{\sc E.~Casas, R.~Herzog, and G.~Wachsmuth}, {\em Optimality conditions and
  error analysis of semilinear elliptic control problems with {$L^1$} cost
  functional}, SIAM J. Optim., 22 (2012), pp.~795--820.

\bibitem{MR3601024}
\leavevmode\vrule height 2pt depth -1.6pt width 23pt, {\em Analysis of
  spatio-temporally sparse optimal control problems of semilinear parabolic
  equations}, ESAIM Control Optim. Calc. Var., 23 (2017), pp.~263--295.

\bibitem{MR3612174}
{\sc E.~Casas and K.~Kunisch}, {\em Stabilization by sparse controls for a
  class of semilinear parabolic equations}, SIAM J. Control Optim., 55 (2017),
  pp.~512--532.

\bibitem{MR3669663}
{\sc E.~Casas, M.~Mateos, and A.~R\"osch}, {\em Finite element approximation of
  sparse parabolic control problems}, Math. Control Relat. Fields, 7 (2017),
  pp.~393--417.

\bibitem{MR3780467}
\leavevmode\vrule height 2pt depth -1.6pt width 23pt, {\em Improved
  approximation rates for a parabolic control problem with an objective
  promoting directional sparsity}, Comput. Optim. Appl., 70 (2018),
  pp.~239--266.

\bibitem{CiarletBook}
{\sc P.~G. Ciarlet}, {\em The finite element method for elliptic problems},
  SIAM, Philadelphia, PA, 2002.

\bibitem{CKW2016}
{\sc C.~Clason, B.~Kaltenbacher, and D.~Wachsmuth}, {\em Functional error
  estimators for the adaptive discretization of inverse problems}, Inverse
  Problems, 32 (2016), pp.~104004, 25.

\bibitem{CK2011}
{\sc C.~Clason and K.~Kunisch}, {\em A duality-based approach to elliptic
  control problems in non-reflexive {B}anach spaces}, ESAIM Control Optim.
  Calc. Var., 17 (2011), pp.~243--266.

\bibitem{CK2012}
\leavevmode\vrule height 2pt depth -1.6pt width 23pt, {\em A measure space
  approach to optimal source placement}, Comput. Optim. Appl., 53 (2012),
  pp.~155--171.

\bibitem{DDP1999}
{\sc E.~Dari, R.~G. Dur\'an, and C.~Padra}, {\em Maximum norm error estimators
  for three-dimensional elliptic problems}, SIAM J. Numer. Anal., 37 (2000),
  pp.~683--700.

\bibitem{DG2012}
{\sc A.~Demlow and E.~H. Georgoulis}, {\em Pointwise a posteriori error control
  for discontinuous {G}alerkin methods for elliptic problems}, SIAM J. Numer.
  Anal., 50 (2012), pp.~2159--2181.

\bibitem{Demlow2016}
{\sc A.~Demlow and N.~Kopteva}, {\em Maximum-norm a posteriori error estimates
  for singularly perturbed elliptic reaction-diffusion problems}, Numer. Math.,
  133 (2016), pp.~707--742.

\bibitem{MR2754849}
{\sc A.~Demlow and R.~Stevenson}, {\em Convergence and quasi-optimality of an
  adaptive finite element method for controlling {$L_2$} errors}, Numer. Math.,
  117 (2011), pp.~185--218.

\bibitem{eriksson1994}
{\sc K.~Eriksson}, {\em An adaptive finite element method with efficient
  maximum norm error control for elliptic problems}, Math. Models Methods Appl.
  Sci., 4 (1994), pp.~313--329.

\bibitem{Guermond-Ern}
{\sc A.~Ern and J.-L. Guermond}, {\em Theory and practice of finite elements},
  vol.~159 of Applied Mathematical Sciences, Springer-Verlag, New York, 2004.

\bibitem{hinze2005}
{\sc M.~Hinze}, {\em A variational discretization concept in control
  constrained optimization: the linear-quadratic case}, Comput. Optim. Appl.,
  30 (2005), pp.~45--61.

\bibitem{nochetto1995}
{\sc R.~Nochetto}, {\em Pointwise a posteriori error estimates for elliptic
  problems on highly graded meshes}, Math. Comp., 64 (1995), pp.~1--22.

\bibitem{NSSV2006}
{\sc R.~H. Nochetto, A.~Schmidt, K.~G. Siebert, and A.~Veeser}, {\em Pointwise
  a posteriori error estimates for monotone semi-linear equations}, Numer.
  Math., 104 (2006), pp.~515--538.

\bibitem{NSV2003}
{\sc R.~H. Nochetto, K.~G. Siebert, and A.~Veeser}, {\em Pointwise a posteriori
  error control for elliptic obstacle problems}, Numer. Math., 95 (2003),
  pp.~163--195.

\bibitem{PV2013}
{\sc K.~Pieper and B.~Vexler}, {\em A priori error analysis for discretization
  of sparse elliptic optimal control problems in measure space}, SIAM J.
  Control Optim., 51 (2013), pp.~2788--2808.

\bibitem{Stadler2009}
{\sc G.~Stadler}, {\em Elliptic optimal control problems with {$L^1$}-control
  cost and applications for the placement of control devices}, Comput. Optim.
  Appl., 44 (2009), pp.~159--181.

\bibitem{MR0192177}
{\sc G.~Stampacchia}, {\em Le probl\`eme de {D}irichlet pour les \'equations
  elliptiques du second ordre \`a coefficients discontinus}, Ann. Inst. Fourier
  (Grenoble), 15 (1965), pp.~189--258.

\bibitem{ver1989}
{\sc R.~Verf{\"u}rth}, {\em A posteriori error estimators for the {S}tokes
  equations}, Numer. Math., 55 (1989), pp.~309--325.

\bibitem{ver2013}
{\sc R.~Verf\"urth}, {\em A posteriori error estimation techniques for finite
  element methods}, Numerical Mathematics and Scientific Computation, Oxford
  University Press, Oxford, 2013.

\bibitem{WW:11}
{\sc G.~Wachsmuth and D.~Wachsmuth}, {\em Convergence and regularization
  results for optimal control problems with sparsity functional}, ESAIM Control
  Optim. Calc. Var., 17 (2011), pp.~858--886.

\end{thebibliography}

\end{document}